\documentclass[11pt]{article}
\usepackage[tbtags]{amsmath}
\usepackage{cases}
\usepackage{mathrsfs}
\usepackage{amsfonts}
\usepackage{amsmath,amsthm,amssymb}

\newcommand{\epi}{\epsilon}
\newcommand{\Dv}{{\rm div}}

\theoremstyle{plain}
\newtheorem{theorem}{Theorem}[section]
\newtheorem{definition}{Definition}[section]
\newtheorem{lemma}{Lemma}[section]
\newtheorem{remark}{Remark}[section]
\newtheorem{proposition}{Proposition}[section]

\newif \ifLastSection \LastSectionfalse

\numberwithin{equation}{section}
\newcommand{\T}{{\mathbb T}}
\newcommand{\R}{{\mathbb R}}

\newcommand{\eps}{\epsilon}
\newcommand{\beq}{\begin{equation}}
\newcommand{\eeq}{\end{equation}}
\newcommand{\bp}{\begin{proof}}
\newcommand{\ep}{\end{proof}}

\begin{document}

\title{{\bf {Existence of global weak solutions for   the   Navier-Stokes-Vlasov-Boltzmann equations}}}

\author{
\begin{tabular}{cc}
&
 L{\sc ei} Y{\sc ao}$^{1}$,\ \ \
C{\sc heng} Y{\sc u$^{2}$ }\ \
\\
&
  {\small\it 1.
       School of Mathematics and Center for Nonlinear Studies,}\\
&
  {\small\it
     Northwest University, Xi'an 710127,  P.R. China }\\
      &{\small\it
        E-mail: yaolei1056@hotmail.com}\\
 &
  {\small\it 2.
      Department of Mathematics, }\\
&  {\small\it
      The University of Texas at Austin, TX 78712, USA}\\
        &
        {\small\it yucheng@math.utexas.edu}
\end{tabular}
}

\date{}

\maketitle

\textbf{{\bf Abstract:}}
A moderately thick spray can be described by  a coupled system  of equations consisting of the incompressible Navier-Stokes equations and the Vlasov-Boltzmann equation.
We investigate this kind of mathematical model in this paper. In particular, we study the initial value problem for  the Navier-Stokes-Vlasov-Boltzmann equations.
The existence of global weak solutions
is established by a weak convergence method.
 The interesting point of our main result is to handle the model with some breakup effects while the velocity of particles is in the whole space.

\bigbreak \textbf{{\bf Key Words}:} The moderately thick spray,  breakup Kernel, Navier-Stokes-Vlasov-Boltzmann equations, Existence, Weak solution.

\bigbreak  {\textbf{AMS Subject Classification :} (2010)  } 35Q30, 35Q99.

\section{Introduction}

A spray is a process in which drops are dispersed in a gas. Sprays can help people to distribute material over a cross-section and to generate liquid surface area.
 There are various applications in different fields, the examples include but are not  limited to  fuel injectors for gasoline and Diesel engines, atomizers for jet engines, atomizers for injecting heavy fuel oil into combustion air in steam boiler injectors, and rocket engine injectors.
 It also has a big impact on crop yields, plant health, efficiency of pest control and of course, profitability. Here we will focus on a model of a moderately thick spray. In general, we refer the reader to \cite{Caflisch} for the physical background.

For a moderately thick spray, we can assume that the volume fraction occupied by the droplets is small enough to be
neglected. Thus, we are able to apply the Vlasov-Boltzmann equation to model the liquid phase, which can be performed by the use of a particle density function. In particular, a function
$f(t,x,\xi)$ denotes a number density of droplets of which at time $t$ and physical position $x$ with velocity $\xi$, which is a solution of the following Vlasov-Boltzmann equation
\begin{equation}
\label{vlasov equation}
f_t+\xi\cdot\nabla_x f+\nabla(f\mathfrak{F})=Q(f,f),\end{equation}
where $\mathfrak{F}$ is an acceleration resulting from the drag force exerted by the gas,
and $Q(f,f)$ is an operator taking into account complex phenomena happening at the level of the droplets
(collisions, coalescences, breakup).

The incompressible Navier-Stokes equations could be used to describe the motion of gas
when instigating  the transport of sprays in the upper airways of the human lungs:
\begin{equation}
\begin{split}&
\label{NS equation}
u_t+u\cdot\nabla u+\nabla P-\mu\Delta u=F_e,
\\&\Dv u=0
\end{split}
\end{equation}
where $u$ is the velocity of the  gas, and $F_e$ denotes the external force with \begin{equation}
\label{external force}
F_e=-c\int_{\R^3}f\mathfrak{F}\,d\xi,
\end{equation}
$\mathfrak{F}$ denotes the acceleration, it is given by
\begin{equation*}
\label{accelation}
\mathfrak{F}=-\frac{9\mu}{2\rho}\frac{\xi-u}{r^2},
\end{equation*} $\mu$ is the viscosity constant of the incompressible Navier-Stokes equations, $\rho$ is the density of gas, $r$ is the radius of the droplet. We can assume that
$\frac{9\mu}{2\rho r^2}=1$ in the whole paper, thus
\begin{equation}
\label{accelation-}
\mathfrak{F}=-(\xi-u).
\end{equation}

\noindent One of the typical forms of the collision kernel is given by
\begin{equation}
\label{operator of collision}
Q(f,f)=-\lambda f(t,x,\xi)+\lambda\int_{\R^3} T(\xi,\xi')f(t,x,\xi')\,d\xi'.
\end{equation}
The constant $\lambda>0$ is  the breakup frequency.
 The kernel $T(\xi,\xi^\prime)$ is the probability of a change with respect to velocity from $\xi^\prime$ to $\xi$, and $\xi$ is the velocity of individuals before the collision while $\xi^\prime$ is the velocity immediately after the collision.
  Given that a reorientation occurs, the probability function $T(\xi,\xi^\prime)$ is a non-negative function and after normalization we have
\begin{equation}\label{E1.4}
\int_{\mathbb{R}^3}T(\xi,\xi^\prime)d\xi=1.
\end{equation}

By the acceleration of the Vlasov-Boltzmann equations and the external force of the incompressible Navier-Stokes equations, the above equations \eqref{vlasov equation}-\eqref{operator of collision} can be coupled with each other. As a result, we arrive at
a  model for a moderately thick spray, namely the Navier-Stokes-Vlasov-Boltzmann equations:
\begin{equation}\label{E1.1}
\left\{\begin{array}{l}
\partial_t f
  +
  \xi\cdot \nabla_x f+{\rm div}_\xi((u-\xi)f)=-\lambda f+\lambda\int_{\mathbb{R}^3}T(\xi,\xi^\prime)f(t,x,\xi^\prime)d\xi^\prime,\\[2mm]
\partial_t u
 +u\cdot \nabla_x u +\nabla_x P-\mu\Delta_x u=-\int_{\mathbb{R}^3}(u-\xi)fd\xi,   \\[2mm]
 {\rm div}u=0,
 \end{array}
        \right.
\end{equation}
with $$\int_{\R^3} T(\xi,\xi')\,d\xi=1.$$
The main goal of this paper is to investigate the existence of weak solutions globally in time $t$ for the equations \eqref{E1.1} with the following initial data
\begin{equation}\label{E1.2}
(f,u)|_{t=0}=(f_0(x,\xi),u_0(x)), \ \ \  x\in \mathbb{T}^3,\ \xi\in \mathbb{R}^3,
\end{equation}
and the initial data $f_0(x,\xi)$ is suitable decay condition as $|\xi|\to\infty.$
\bigbreak

The mathematical analysis for spray models is very challengingly because the coupled term for unknowns that does not depend on the same set of variables.
  The existence theory of global weak solutions for such models,  dates back to the late of 1990's. The first work in this field was  \cite{Hamdache} where the author proved the global  existence of weak solution and their large-time behavior for the Vlasov-Stokes equations. The existence theorem for weak solutions was extended in \cite{Ano, Boudin, Mellet1, Yu2, Yu1}, where the authors did not neglect the convection term and considered the Navier-Stokes equations, including incompressible and compressible ones. In \cite{Goudon1}, the existence and uniqueness of global smooth solutions near an equilibrium was proved under smallness conditions for the Navier-Stokes system coupled with the Vlasov-Fokker-Planck equation in 3D. In the meantime,  there has been  a lot of work  in hydrodynamic limits; we refer the reader to \cite{Goudon2, Goudon3, Mellet2}. In these works, the authors rely on convergence methods, such as the  compactness and relative entropy methods, to investigate hydrodynamic limits. It is natural to consider the mathematical analysis of  models with
  effects of collision or the effects of breakup, are but more challenging.
   In \cite{Legar}, Legar-Vasseur established the existence theory for  weak solutions for a system involving the coupling of  the Navier-Stokes equations and the Vlasov-Boltzmann equation where they restricted the velocity $\xi$ to  a bounded domain.  Benjelloun-Desvillettes-Moussa in \cite{Benjelloun}
   introduced  the Navier-Stokes-Vlasov-Boltzman model for spray theory, and a typical operator is given when the droplets after breakup have the same velocities as before breakup.
   Finally, they
   established an existence result of weak solutions for a simple model which was derived from the Navier-Stokes-Vlasov-Boltzman model. In particular, they
assumed that the aerosol is bidispersed, in other words,  there are only two possible radius $r_1 > r_2$ exist for the droplets, and the result
for  the breakup of particles of radii $r_1$ are particles of radius $r_2$. Under this assumption, the density function will be split up as follows $$ f(t, x, \xi, r) = f_1(t, x, \xi)\delta_{r=r_1} + f_2(t, x, \xi)\delta_{r=r_2},$$
therefore, the  Navier-Stokes-Vlasov-Boltzman model will reduce to the Navier-Stokes-Vlasov equations with damping terms. Motivated by the work of Yu \cite{Yu}, of particular interests in this paper is to establish the existence of weak solutions of
 Navier-Stokes-Vlasov-Boltzman equations with a typical breakup operator given in \cite{Benjelloun}, for the velocity $\xi\in \R^3.$ However, the restriction of the same velocities is not necessary in our paper.
  \\

It is not necessary to assume that the droplets after breakup have the same velocities as before breakup in this paper. In fact,  we deduce the following  from the conservation of kinetic energy,
\begin{equation*}
|\xi|^2=|\xi^\prime|^2,
\end{equation*}
which implies
\begin{equation}\label{E1.3}
|\xi|=|\xi^\prime|.
\end{equation}
 We will assume \eqref{E1.3} in the whole paper. In other words, different from the work of \cite{Benjelloun}, we only need to assume that the droplets after breakup have the same speeds as before breakup, not the same velocities.

\noindent As \cite{Legar}, we assume that $T(\xi,\xi^\prime)$ satisfies a self-similarity property, namely,
\begin{equation}\label{E1.5}
T(\xi,\xi^\prime)=H(|\xi^\prime|)T(\frac{\xi}{|\xi^\prime|},\frac{\xi^\prime}{|\xi^\prime|}), \textrm{\ for\  some\  function}\  H(\cdot).
\end{equation}

\noindent Note that, smooth solutions of the   problem \eqref{E1.1}-\eqref{E1.2} satisfy an  energy equality. In particular, we have the energy inequality
\begin{equation}\label{E1.6}
\begin{split}
&\frac{1}{2}\int_{\mathbb{T}^3}|u|^2dx+\int_{\mathbb{T}^3}\int_{\mathbb{R}^3}f(1+\frac{1}{2}|\xi|^2)d\xi dx +\int_0^T \int_{\mathbb{T}^3}|\nabla u|^2dx dt+\int_0^T \int_{\mathbb{T}^3}\int_{\mathbb{R}^3}f|u-\xi|^2d\xi dx  dt\\
=&\frac{1}{2}\int_{\mathbb{T}^3}|u_0|^2dx+\int_{\mathbb{T}^3}\int_{\mathbb{R}^3}f_0(1+\frac{1}{2}|\xi|^2)d\xi dx ,
\end{split}
\end{equation}
for any $T>0$.   Thus, it is natural to suppose that the initial data satisfy
\begin{equation}\label{E1.7}
\int_{\mathbb{T}^3}|u_0|^2dx+\int_{\mathbb{T}^3}\int_{\mathbb{R}^3}f_0(1+\frac{1}{2}|\xi|^2)d\xi dx <\infty.
\end{equation}

\bigbreak

Based on the energy equality \eqref{E1.6}, we define the concept of weak solution for the   problem \eqref{E1.1}-\eqref{E1.2} as follows:\\
\begin{definition}\label{Def1.1}
A pair $(u,f)$ is called a  global weak solution to the problem \eqref{E1.1}-\eqref{E1.2} if, for any $T>0$, the following properties hold:
\begin{itemize}
\item $u\in L^\infty(0,T; L^2(\mathbb{T}^3))\cap L^2(0,T; H^1(\mathbb{T}^3))$;
\item $f(t,x,\xi)\geq 0$, for any $(t,x,\xi)\in(0,T)\times\mathbb{T}^3\times\mathbb{R}^3$;
\item $f\in L^\infty(0,T; L^\infty(\mathbb{T}^3\times\mathbb{R}^3)\cap L^1(\mathbb{T}^3\times\mathbb{R}^3) )$;
\item $|\xi|^3f \in L^\infty(0,T; L^1(\mathbb{T}^3\times\mathbb{R}^3))$;
\item for any test function $\varphi\in C^\infty([0,T]\times \mathbb{T}^3)$ with ${\rm div}\varphi=0$, we have
\begin{equation}\label{E1.8}
\begin{split}&
\int_0^T\int_{\mathbb{T}^3}(-u\cdot\varphi_t+(u\cdot \nabla) u\cdot\varphi+\nabla u: \nabla \varphi)dxd
\\&\quad\quad\quad\quad=-\int_0^T\int_{\mathbb{T}^3}\int_{\mathbb{R}^3}f(u-\xi)\cdot \varphi d\xi dx  dt+\int_{\mathbb{T}^3}u_0\cdot\varphi(0,x)dx;
\end{split}
\end{equation}
\item for any test function $\phi\in C^\infty([0,T]\times\mathbb{T}^3\times\mathbb{R}^3)$, we have
\begin{equation}\label{E1.9}
\begin{split}
&-\int_0^T\int_{\mathbb{T}^3}\int_{\mathbb{R}^3}f(\phi_t+\xi\cdot \nabla_x \phi +(u-\xi)\cdot \nabla_{\xi}\phi)d\xi dx dt=\int_{\mathbb{T}^3}\int_{\mathbb{R}^3}f_0\phi(0,x,\xi)d\xi dx\\
&+\lambda\int_0^T\int_{\mathbb{T}^3}\int_{\mathbb{R}^3}f\phi d\xi dx  dt
-\lambda\int_0^T\int_{\mathbb{T}^3}\int_{\mathbb{R}^3}\int_{\mathbb{R}^3}T(\xi,\xi^\prime)f(t,x,\xi^\prime)\phi d\xi^\prime d\xi  dx dt;
\end{split}
\end{equation}
\item the energy inequality
\begin{equation*}
\begin{split}
&\frac{1}{2}\int_{\mathbb{T}^3}|u|^2dx+\int_{\mathbb{T}^3}\int_{\mathbb{R}^3}f(1+\frac{1}{2}|\xi|^2)d\xi dx+\int_0^T \int_{\mathbb{T}^3}|\nabla_x u|^2dx dt\\
&+\int_0^T \int_{\mathbb{T}^3}\int_{\mathbb{R}^3}f|u-\xi|^2d\xi dx  dt\\
\leq &\frac{1}{2}\int_{\mathbb{T}^3}|u_0|^2dx+\int_{\mathbb{T}^3}\int_{\mathbb{R}^3}f_0(1+\frac{1}{2}|\xi|^2)d\xi dx,
\end{split}
\end{equation*}
for any $t\in [0,T].$
\end{itemize}
\end{definition}

Our main results in this paper is as follows:
\begin{theorem}\label{Thm 1}
If the initial data satisfy $\Dv u_0=0$, and \eqref{E1.7}, the probability function $T(\xi,\xi^\prime)$ satisfies \eqref{E1.4} and \eqref{E1.5},
then there exists a global weak solution to the problem \eqref{E1.1}-\eqref{E1.2}.
\end{theorem}

\begin{remark}\label{Rem 1.1}
The same result as in Theorem \ref{Thm 1} holds for the initial boundary value problem \eqref{E1.1}-\eqref{E1.2}, with the boundary condition
$u=0$ on $\partial\Omega$ and $f(t,x,\xi)=f(t,x,\xi^*)$ for $x\in \partial\Omega$, $\xi\cdot n(x)<0$, where $\xi^*=\xi-2(\xi\cdot n(x))n(x)$ is the specular
velocity, $n(x)$ is the outward normal to $\Omega,$ and $\Omega\subset\mathbb{R}^3$ is a bounded domain.
Meanwhile,
the incompressible Navier-Stokes equations in the system \eqref{E1.1} can be replaced by the inhomogeneous ones; the extension of our result, in this
context, is considered in the forthcoming paper \cite{Yao}.
\end{remark}

 An interesting aspect of our main result is that it  handles the model with some breakup effects while the velocity of particles is in the whole space. As we mentioned before, we need to assume that the droplets after breakup have the same speeds as before breakup, but not the same velocities. This is different from the work of  \cite{Benjelloun}.
A key observation of our proof is Lemma \ref{Lem2.2}, which gives us some uniform control on the density function $f(t,x,\xi)$. By  Fubini's Theorem, we are able to show the following {\it a priori} estimate
\begin{equation}
\label{operator 0}
\int_{\mathbb{T}^3}\int_{\R^{3}}|\xi|^p Q(f) \,d\xi\; dx=0\;\text{ for any } p\geq1,
\end{equation}
which allows us to obtain further bounds and compactness of smooth solutions.  Thus, with a suitable approximation, a weak solution could be recovered. Our idea of the approximation is to construct an iteration for the kinetic part, and to adopt
a  Galerkin  method for the fluid part. However, the breakup operator in the  iteration is given by
\begin{equation}
\label{operator iteration}
-\lambda f^n+\lambda\int_{\mathbb{R}^3}T(\xi,\xi^\prime)f^{n-1}(t,x,\xi^\prime)d\xi^\prime,
\end{equation}
for any integer $n\geq 1.$
Note that, $\{f^n\}$ is an increasing sequence. This allows us to obtain the bounds on $f^n$ and $\int|\xi|^kf ^n\,d\xi$, which yields   weak stability. By the weak convergence method, the existence of weak solutions can be  deduced.

{\bf Notations}: In the following, $C$ from line to line denote a constant  depending   on the initial data, $T$ and the physical coefficients; $C(E,B)$    denotes  a generic   positive constant  depending   on the initial data, $T$, the physical coefficients and $E$, $B$.

We organize the rest of the paper as follows. In Section 2, we deduce a priori estimates, state some useful lemmas. In Section 3, we construct a smooth solution of an approximation scheme for  \eqref{E1.1}-\eqref{E1.2}. In Section 4, we recover   weak solutions from the approximations  by a  weak convergence method.

\bigskip

\section{{\it A  priori} estimates and some useful lemmas}\label{S2}

In this section, we derive a {\it priori} estimates for   the problem \eqref{E1.1}-\eqref{E1.2}, which will help us to derive the weak stability of the solutions. Firstly, we derive an  energy inequality for any smooth solution  of \eqref{E1.1}-\eqref{E1.2}.

\begin{lemma}\label{Lem 2.1}
For any smooth solution $(u,f)$ to the  problem \eqref{E1.1}-\eqref{E1.2},  the following equality holds
\begin{equation}
\begin{split}
\label{energy equality}
&\frac{1}{2}\int_{\mathbb{T}^3}|u|^2dx+\int_{\mathbb{T}^3}\int_{\mathbb{R}^3}f(1+\frac{1}{2}|\xi|^2)d\xi dx+\int_0^T \int_{\mathbb{T}^3}|\nabla u|^2dx dt+\int_0^T\int_{\mathbb{T}^3}\int_{\mathbb{R}^3}f|u-\xi|^2d\xi dx dt\\
=& \frac{1}{2}\int_{\mathbb{T}^3}|u_0|^2dx+\int_{\mathbb{T}^3}\int_{\mathbb{R}^3}f_0(1+\frac{1}{2}|\xi|^2)d\xi dx.
\end{split}
\end{equation}
\end{lemma}
\begin{proof}
Taking the scalar product with  $u$ on both sides of \eqref{E1.1}$_2$, and integrating over $\mathbb{T}^3$, we have
\begin{equation}\label{E2.1}
\frac{1}{2}\frac{d}{dt}\int_{\mathbb{T}^3}|u|^2dx +\int_{\mathbb{T}^3}|\nabla u|^2dx=-\int_{\mathbb{T}^3}\int_{\mathbb{R}^3}f(u-\xi)\cdot ud\xi dx.
\end{equation}
Multiplying by $1+\frac{1}{2}|\xi|^2$ on both sides of \eqref{E1.1}$_1$, and integrating over $\mathbb{T}^3\times \mathbb{R}^3$, we have
\begin{equation}\label{E2.2}
\begin{split}
&\frac{d}{dt}\int_{\mathbb{T}^3}\int_{\mathbb{R}^3}f(1+\frac{1}{2}|\xi|^2)d\xi dx+\int_{\mathbb{T}^3}\int_{\mathbb{R}^3}f|u-\xi|^2d\xi dx\\
=&\int_{\mathbb{T}^3}\int_{\mathbb{R}^3}f(u-\xi)\cdot ud\xi dx-\lambda\int_{\mathbb{T}^3}\int_{\mathbb{R}^3}f(1+\frac{1}{2}|\xi|^2)d\xi dx\\
&+ \lambda\int_{\mathbb{T}^3}\int_{\mathbb{R}^3}\int_{\mathbb{R}^3}T(\xi,\xi^\prime)f(t,x,\xi^\prime)d\xi^\prime(1+\frac{1}{2}|\xi|^2)d\xi dx.
\end{split}
\end{equation}
By using  Fubini's theorem, \eqref{E1.3} and \eqref{E1.4}, the last term on the right-hand side of \eqref{E2.2} can be estimated as follows
\begin{equation}
\begin{split}
\label{aaaaa}
& \lambda\int_{\mathbb{T}^3}\int_{\mathbb{R}^3}\int_{\mathbb{R}^3}T(\xi,\xi^\prime)f(t,x,\xi^\prime)d\xi^\prime(1+\frac{1}{2}|\xi|^2)d\xi dx\\
 =&\lambda\int_{\mathbb{T}^3}\int_{\mathbb{R}^3}\int_{\mathbb{R}^3} T(\xi,\xi^\prime)f(t,x,\xi^\prime)(1+\frac{1}{2}|\xi^\prime|^2)d\xi d\xi^\prime dx\\
= &\lambda\int_{\mathbb{T}^3}\int_{\mathbb{R}^3}  f(t,x,\xi)(1+\frac{1}{2}|\xi|^2)d\xi dx.
\end{split}
\end{equation}
Adding \eqref{E2.1} to \eqref{E2.2} and using \eqref{aaaaa},  \eqref{energy equality} follows.
\end{proof}

\bigbreak

To develop further estimates, we will rely on the following lemma.
\begin{lemma}[\cite{Yu}]\label{Lem2.2}
Assume that $T(\xi,\xi^\prime)$ satisfies \eqref{E1.4} and \eqref{E1.5}. Then  there exists a constant $K>0$, such that
\begin{equation}\label{E2.4}
\int_{\mathbb{R}^3}T(\xi,\xi^\prime)d\xi^\prime\leq K<\infty.
\end{equation}
\end{lemma}
%\begin{proof}
%The proof is motivated by the work of \cite{Legar}. From \eqref{E1.3}, we deduce
%\begin{equation}\label{E2.5}
%T(\xi,\xi^\prime)=0,\ if\ |\xi|>|\xi^\prime|.
%\end{equation}

%By \eqref{E1.4} and \eqref{E2.5}, we have
%\begin{equation}
%\begin{split}
%1=\int_{\mathbb{R}^3}T(\xi,\xi^\prime)d\xi=&\int_{B(0,|\xi^\prime|)}T(\xi,\xi^\prime)d\xi+\int_{|\xi|>|\xi^\prime|}T(\xi,\xi^\prime)d\xi\\
%=& \int_{B(0,|\xi^\prime|)}T(\xi,\xi^\prime)d\xi.
%\end{split}
%\end{equation}
%From \eqref{E1.5}, we have
%\begin{equation}
%\begin{split}
%1=& H(|\xi^\prime|)\int_{B(0,|\xi^\prime|)}T(\frac{\xi}{|\xi^\prime|},\frac{\xi^\prime}{|\xi^\prime|})d\xi\\
%=& H(|\xi^\prime|)\int_{B(0,1)}T(z,\frac{\xi^\prime}{|\xi^\prime|})|\xi^\prime|^3dz\\
%=& |\xi^\prime|^3  H(|\xi^\prime|),
%\end{split}
%\end{equation}
%where $z=\frac{\xi}{|\xi^\prime|}$. This gives
%\begin{equation}
%H(|\xi^\prime|)=\frac{1}{|\xi^\prime|^3 },
%\end{equation}
%and by \eqref{E1.3}, we have
%\begin{equation*}
%\begin{split}
%\int_{\mathbb{R}^3}T(\xi,\xi^\prime)d\xi^\prime=&\int_{B(0,|\xi^\prime|)} H(|\xi^\prime|)T(\frac{\xi}{|\xi^\prime|},\frac{\xi^\prime}{|\xi^\prime|})d\xi^\prime\\
%=& \frac{1}{|\xi|^3 }\int_{B(0,|\xi|)} T(\frac{\xi}{|\xi^\prime|},\frac{\xi^\prime}{|\xi^\prime|})d\xi^\prime\\
%=& \frac{1}{|\xi|^3 }\int_{B(0,1)} T(\frac{\xi}{|\xi|},\eta)|\xi|^3d\eta\\
%=& \int_{B(0,1)} T(\frac{\xi}{|\xi|},\eta)d\eta\leq K.
%\end{split}
%\end{equation*}
%\end{proof}

\bigbreak

\noindent In what follows, we denote
\begin{equation*}
m_\alpha f(t,x)=\int_{\mathbb{R}^3}|\xi|^\alpha fd\xi,\ \ and\   M_\alpha f(t)=\int_{\mathbb{T}^3}\int_{\mathbb{R}^3}|\xi|^\alpha fdxd \xi,
\end{equation*}
here $\alpha\geq 0$ is a constant.
Clearly,
\begin{equation*}
M_\alpha f(t)=\int_{\mathbb{T}^3}m_\alpha f dx.
\end{equation*}

\begin{lemma}[\cite{Hamdache}]\label{Lem2.3}
Let $\beta>0$ and let $f$ be a nonnegative function in $L^\infty((0,T)\times \mathbb{T}^3\times\mathbb{R}^3),$  such that $m_\beta f(t,x)<+\infty,$ for a.e. $(t,x)$. Then the following estimate  holds for any $\alpha<\beta$:
\begin{equation}
m_\alpha f(t,x)\leq C(\|f(t,x,\cdot)\|_{L^\infty(\mathbb{R}^3)}+1)m_\beta f(t,x)^{\frac{\alpha+3}{\beta+3}}, \ a.e. \ (t,x).
\end{equation}
\end{lemma}

%Denote $\overline{g}(t,x,\xi)=g*\theta_\epi(x),$  where we define $\theta_\epsilon$ a mollifier such that $\theta_\epsilon(x)=\epsilon^{-3}\theta(\frac{x}{\epsilon})$ with $\theta\in C^\infty(\mathbb{R}^3)$, $\theta\geq 0$, and $\int_{\mathbb{R}^3}\theta dx=1$.
%We give the following lemma, which will be useful.
%\begin{lemma}[\cite{Yu}]\label{Lem2.7}
%For any function $h=h(\xi),$ we have
%\beq\notag
%\int_{\R^3}\overline{g}h(\xi)d\xi=\overline{\int_{\R^3}ghd\xi}.
%\eeq
%\end{lemma}
\section{An approximation scheme}\label{S3}
In this  section, we construct  smooth solutions of an approximation system.   For that purpose, we define a finite-dimensional space
$
X_m=\textrm{span}\{\phi_i\}_{i=1}^{m},
$
where  $\{\phi_i\}_{i=1}^{m}\subset C_0^\infty(\T^3)$ is an orthonormal  basis of $\{v\in L^2(\T^3): \textrm{div}v=0 \  \textrm{in}\  \mathcal{D}^\prime\}.$ Define $Y_m=C([0,T];X_m). $

We propose the following approximation scheme associated with  the  Navier-Stokes-Boltzmann equations \eqref{E1.1}
\begin{equation}\label{E3.1}
\left\{\begin{array}{l}
\partial_t f_m
  +
  \xi\cdot \nabla_x f_m+{\rm div}_\xi((\tilde{u}-\xi)f_m)=-\lambda f_m+\lambda\int_{\mathbb{R}^3}T(\xi,\xi^\prime)f_m(t,x,\xi^\prime)d\xi^\prime,\\[2mm]
\partial_t u_m
 +(\tilde{u}\cdot \nabla_x)u_m+\nabla_x P_m-  \Delta_x u_m=-\int_{\mathbb{R}^3}(\tilde{u}-\xi)f_m  d\xi,    \\[2mm]
 {\rm div}u_m=0,  \ \ \ \   \qquad x\in \mathbb{T}^3,\ \xi\in \mathbb{R}^3, \  t\in(0,T), \\[2mm]
 \end{array}
        \right.
\end{equation}
with initial data
\begin{equation}\label{E3.2}
(f_m,u_m)|_{t=0}=(f_0^\epsilon(x,\xi),u_0^\epsilon(x)), \ \ \  x\in \mathbb{T}^3,\ \xi\in \mathbb{R}^3,
\end{equation}
where $\tilde{u}$ is given in $Y_m$. The initial data  $f_0^\epsilon$ and $u_0^\epsilon$ are $C^\infty$ functions  such that  $f_0^\epsilon\rightarrow f_0$ strongly in $L^p(\mathbb{T}^3\times \mathbb{R}^3)$, for all $p<\infty$, and weakly in $weak\ast$-$L^\infty(\mathbb{T}^3\times \mathbb{R}^3)$;$f_0^\epsilon$ has a compact support with respect to $\xi$ in  $\R^3$,  $M_3 f_0^\epsilon$ is uniformly bounded with respect to $\epi,$  and $u_0^\epsilon\rightarrow u_0$ strongly in $L^2(\mathbb{T}^3)$.

\subsection{The weak solutions of kinetic part}
\noindent The first step of solving \eqref{E3.1}-\eqref{E3.2} is to investigate the global existence of weak solutions of  the following problem:
\begin{equation}\label{E3.3-1}
\left\{\begin{array}{l}
\partial_t f_m
  +
  \xi\cdot \nabla_x f_m+{\rm div}_\xi((\tilde{u}-\xi)f_m)=-\lambda f_m+\lambda\int_{\mathbb{R}^3}T(\xi,\xi^\prime)f_m(t,x,\xi^\prime)d\xi^\prime.\\[2mm]
f_m|_{t=0}=f_0^\epsilon. \\[2mm]
 \end{array}
        \right.
\end{equation}
To this end,
we construct a sequence of solutions in $n$ verifying
\begin{equation}\label{E3.4-1}
\left\{\begin{array}{l}
\partial_t f_m^n
  +
  \xi\cdot \nabla_x f_m^n+{\rm div}_\xi((\tilde{u}-\xi)f_m^n)=-\lambda f_m^n+\lambda\int_{\mathbb{R}^3}T(\xi,\xi^\prime)f_m^{n-1}(t,x,\xi^\prime)d\xi^\prime,\\[2mm]
f_m^n|_{t=0}=f_0^\epsilon, \\[2mm]
f_m^0=0.
 \end{array}
        \right.
\end{equation}
The method of characteristics gives us a solution of \eqref{E3.4-1} for any given $f_m^{n-1}$. Indeed, when  $f_m^{n-1}$ is given, we are able to define the trajectories $x(\tau)=x(\tau,t,x,\xi)$ and $\xi(\tau)=\xi(\tau,t,x,\xi)$ with the following ODE system
\begin{equation}\notag
\left\{\begin{split}
&
\frac{d x}{d\tau}(\tau)=\xi(\tau),\\
&
\frac{d\xi}{d\tau}(\tau)= \tilde{u}(t, x(\tau))-\xi(\tau),\\
&
x(t,t,x,\xi)=x,\\
&
\xi(t,t,x,\xi)=\xi.
\end{split}
\right.
\end{equation}
Along the trajectories above, the solutions of \eqref{E3.4-1} satisfy
\begin{equation}\label{E3.3}
\frac{d}{d\tau}f^n_m(\tau,x(\tau),\xi(\tau))=(3-\lambda)f^n_m(\tau,x(\tau),\xi(\tau))+\lambda\int_{\mathbb{R}^3}T(\xi(\tau),\xi^\prime)f^{n-1}_m(\tau,x(\tau),\xi^\prime)d\xi^\prime.
\end{equation}
By   standard theory of ODES,  there exists a smooth solution of \eqref{E3.3} as follows
\begin{equation}\label{E3.4}
\begin{split}
&f^n_m(t,x,\xi)
=e^{(3-\lambda)t}f_0^\epsilon(x(0,t,x,\xi),\xi(0,t,x,\xi))\\
&+\lambda\int_0^t\int_{\mathbb{R}^3}e^{(3-\lambda)(t-\tau)}T(\xi(\tau,t,x,\xi),\xi^\prime)f^{n-1}_m(\tau,x(\tau,t,x,\xi),\xi^\prime)d\xi^\prime d\tau.
\end{split}
\end{equation}
It is also a smooth solution of \eqref{E3.4-1} for any given $f_m^{n-1}$.
Thanks to \eqref{E3.4}, we find
\beq\notag
f^n_m\geq 0, \ {\rm for}\ {\rm all} \  n\geq 0,
\eeq
and $\{f^n_m\}_{n=1}^{\infty} $ is an increasing sequence of measure functions with respect to $n$. For this proof, we refer the reader  to \cite{Legar}.\\

Next we shall apply a compactness argument to recover  weak solutions of problem \eqref{E3.3-1} by passing to the limit in $f_m^n$ as $n$ goes to infinity.
To this end, we need to derive uniform bounds on $f^{n}_m$ with respect to $n$. Note that, $\{f^n_m\}_{n=1}^{\infty} $ is   increasing sequence of nonnegative functions. The Gronwall inequality yields a uniform estimates of $f_m^n$ with respect to $n$.\\

The following Lemma \ref{Lem3.1} provides that $f_m^n$ is bounded in $ L^\infty(0,T;L^p(\T^3\times\R^3)),\text{ for any } p\geq 1,$ with respect to $n>0.$
\begin{lemma}\label{Lem3.1}
For any $n\geq 0$ and fixed $m>0$, $f^n_m(t,x,\xi)$ is bounded in
\beq\notag
L^\infty(0,T;L^\infty(\T^3\times\R^3))\cap L^\infty(0,T;L^1(\T^3\times\R^3)),
\eeq
and hence
\beq\label{E3.6}
f^n_m(t,x,\xi)\ is \ bounded\ in\ L^\infty(0,T;L^p(\T^3\times\R^3)),\text{ for any } p\geq 1.
\eeq
\end{lemma}
\bp
Thanks to Lemma \ref{Lem2.2}, we can deduce the following bound  from  \eqref{E3.4}
\beq\label{E3.7}
\begin{split}
\|f^n_m(t,x,\xi)\|_{L^\infty}\leq & e^{|3-\lambda |T}\|f^\epi_0\|_{L^\infty}+\lambda e^{|3-\lambda |T}K\int_0^t\|f^{n-1}_m(\tau,x,\xi)\|_{L^\infty}d\tau\\
\leq & e^{|3-\lambda |T}\|f_0\|_{L^\infty}+\lambda e^{|3-\lambda |T}K\int_0^t\|f^{n}_m(\tau,x,\xi)\|_{L^\infty}d\tau,\\
\end{split}
\eeq
where we used that $\{f_m^n\}$ is an increasing sequence.

\noindent Applying the Gronwall inequality to \eqref{E3.7}, one obtains
\beq\notag
\|f^n_m\|_{L^\infty}\leq e^{|3-\lambda |T}\|f_0\|_{L^\infty}(1+\lambda e^{|3-\lambda |T}Kte^{\lambda e^{|3-\lambda |T}Kt}),\ for\ any\  0\leq t\leq T.
\eeq
Next, integrating \eqref{E3.3} over $(0,t)\times \T^3\times \R^3$, we find
\beq\notag
\begin{split}
\int_{\T^3}\int_{\R^3}f^n_m(t,x,\xi)d\xi dx=&\int_{\T^3}\int_{\R^3}f^\epi_0(x(0,t,x,\xi),\xi(0,t,x,\xi))d\xi dx\\
 &+(3-\lambda) \int_0^t\int_{\T^3}\int_{\R^3}f^n_m(\tau,x(\tau),\xi(\tau))d\xi dx d\tau\\
&+\lambda   \int_0^t\int_{\T^3}\int_{\R^3}\int_{\mathbb{R}^3}T(\xi(\tau),\xi^\prime)f^{n-1}_m(t,x(\tau),\xi^\prime)d\xi^\prime d\xi dx d\tau,
\end{split}
\eeq
which implies
\beq \label{E3.8}
\begin{split}
\int_{\T^3}\int_{\R^3}f^n_m(t,x,\xi)d\xi dx\leq &\int_{\T^3}\int_{\R^3}f^\epi_0(x(0,t,x,\xi),\xi(0,t,x,\xi))d\xi dx\\
 &+\lambda   \int_0^t\int_{\T^3}\int_{\R^3}\int_{\mathbb{R}^3}T(\xi(\tau),\xi^\prime)f^{n-1}_m(\tau,x(\tau),\xi^\prime)d\xi^\prime d\xi dx d\tau\\
 &+ 3\int_0^t\int_{\T^3}\int_{\R^3}f^n_m(\tau,x(\tau),\xi(\tau))d\xi dx d\tau\\
\leq &\int_{\T^3}\int_{\R^3}f_0(x,\xi)d\xi dx +\lambda   \int_0^t\int_{\T^3}\int_{\R^3}\int_{\mathbb{R}^3}T(\xi(\tau),\xi^\prime)f^{n-1}_m(\tau,x(\tau),\xi^\prime)d\xi^\prime d\xi dxd\tau \\
&+ 3\int_0^t\int_{\T^3}\int_{\R^3}f^n_m(\tau,x(\tau),\xi(\tau))d\xi dx d\tau,
\end{split}
\eeq
where we have used the facts that $T(\xi,\xi^\prime)\geq 0 $, $
f^n_m\geq 0 \ {\rm for}\ {\rm all} \  n\geq 0,$
 and $\lambda>0.$\\

\noindent Thanks to \eqref{E1.4} and   Fubini's Theorem, we have
\beq \notag
\begin{split}
\int_{\T^3}\int_{\R^3}f^n_m(t,x,\xi)d\xi dx\leq &\int_{\T^3}\int_{\R^3}f_0(x,\xi)d\xi dx +(\lambda +3)  \int_0^t\int_{\T^3}\int_{\R^3}f^{n}_m(\tau,x,\xi) d\xi dx d\tau.
\end{split}
\eeq
Applying  the Gronwall inequality  again, this yields
\beq\notag
\|f^n_m\|_{L^\infty(0,T;L^1(\T^3\times \R^3))}\leq \lambda e^{|3-\lambda |T}K(1+(\lambda+3)t e^{(\lambda+3)t}),\ for\ any\  0\leq t\leq T.
\eeq
\ep
\noindent Next we shall show that $\int_{\mathbb{R}^3}T(\xi,\xi^\prime)f^{n}_m(t,x,\xi^\prime)d\xi^\prime$ is bounded in $
L^\infty(0,T;L^p(\T^3\times\R^3)),\ for\ any\ p\geq 1$ in the following lemma.
\begin{lemma}\label{Lem4.2}
For any $n\geq 0$, $\int_{\mathbb{R}^3}T(\xi,\xi^\prime)f^{n}_m(t,x,\xi^\prime)d\xi^\prime$ is bounded in
\beq\notag
L^\infty(0,T;L^p(\T^3\times\R^3)),\ for\ any\ p\geq 1.
\eeq
\end{lemma}
\bp
We derive the following bound from \eqref{E2.4} and \eqref{E3.6}:
\beq
\left\|\int_{\mathbb{R}^3}T(\xi,\xi^\prime)f^{n}_m(t,x,\xi^\prime)d\xi^\prime\right\|_{L^\infty}\leq  \|f^{n}_m\|_{L^\infty}\int_{\mathbb{R}^3}T(\xi,\xi^\prime)d\xi^\prime\leq C.
\eeq
Also, by \eqref{E1.4}, \eqref{E3.6} and  Fubini's theorem, we find
\beq
\begin{split}
& \int_{\T^3}\int_{\R^3}\left|\int_{\mathbb{R}^3}T(\xi,\xi^\prime)f^{n}_m(t,x,\xi^\prime)d\xi^\prime\right|d\xi dx\\
=& \int_{\T^3}\int_{\R^3}\int_{\mathbb{R}^3}T(\xi,\xi^\prime)f^{n}_m(t,x,\xi^\prime)d\xi^\prime d\xi dx\\
=& \int_{\T^3}\int_{\R^3}f^{n}_m d\xi dx\leq C,
\end{split}
\eeq
here $C>0$  depends only on the initial data, $\lambda$ and $T$.
This completes the proof of Lemma \ref{Lem4.2}.
\ep

In order to pass to the limit as $n$ goes to infinity,  we need   bounds on $\int_{\R^3} f^n_m d\xi$,  $\int_{\R^3}\xi f^n_m d\xi$  and $\int_{\R^3}|\xi|^2f^n_m d\xi$ stated in the following lemma.
\begin{lemma}\label{Lem4.3-1}
For any $n\geq 0$, if
\beq\notag
M_k f^\epi_0 =\int_{\T^3}\int_{\R^3}|\xi|^kf^\epi_0d\xi dx<+\infty,
\eeq
for some $k\geq 1$,  then the following estimates hold
\beq\label{E4.10-}
\left\|\int_{\R^3} f^n_m d \xi\right\|_{L^\infty(0,T;L^\frac{3+k}{3}(\T^3))}\leq C(\lambda,T,M_k f^\epi_0,\|f^\epi_0\|_{L^\infty},\tilde{u}),
\eeq
\beq\label{E4.10}
\left\|\int_{\R^3}\xi f^n_m d \xi\right\|_{L^\infty(0,T;L^\frac{3+k}{4}(\T^3))}\leq C(\lambda,T,M_k f^\epi_0,\|f^\epi_0\|_{L^\infty},\tilde{u}),
\eeq
and
\beq\label{E4.11}
\left\|\int_{\R^3}|\xi|^2 f^n_m d \xi\right\|_{L^\infty(0,T;L^\frac{3+k}{5}(\T^3))}\leq C(\lambda,T,M_k f^\epi_0,\|f^\epi_0\|_{L^\infty},\tilde{u}).
\eeq
\end{lemma}
\bp

Multiplying by $|\xi|^k$  on both sides of \eqref{E3.3}, then integrating over $\T^3\times\R^3$,     for $k\geq 1$, we have
\beq\label{E4.12-1}
\begin{split}
&\frac{d}{d\tau}\int_{\T^3}\int_{\R^3}|\xi|^kf^n_m(\tau,x(\tau),\xi(\tau)) d\xi dx+\lambda\int_{\T^3}\int_{\R^3}|\xi|^kf^n_m(\tau,x(\tau),\xi(\tau)) d\xi dx\\
=& \lambda \int_{\T^3}\int_{\R^3}\int_{\R^3}T(\xi(\tau),\xi')|\xi|^kf^{n-1}_m(\tau,x(\tau),\xi')\,d\xi' d\xi dx+3 \int_{\T^3}\int_{\R^3}|\xi|^kf^n_m(\tau,x(\tau),\xi(\tau)) d\xi dx\\
\leq &\lambda \int_{\T^3}\int_{\R^3}|\xi|^kf^{n-1}_m(\tau,x(\tau),\xi) d\xi dx +3\int_{\T^3}\int_{\R^3}|\xi|^{k}f^n_m(\tau,x(\tau),\xi(\tau)) d\xi dx.
\end{split}
\eeq
Integrating with respect to time $\tau$ over $(0,t)$, we have
\begin{equation}
\begin{split}
\label{inequality on k order}
&\int_{\T^3}\int_{\R^3}|\xi|^kf^n_m(t,x,\xi) d\xi dx
\\&\leq
\lambda \int_0^T\int_{\T^3}\int_{\R^3}|\xi|^kf^{n}_m d\xi dx\,dt +3\int_0^T\int_{\T^3}\int_{\R^3}|\xi|^{k}f^n_m d\xi dx\,dt
\\&+\int_{\T^3}\int_{\R^3}|\xi|^kf_0^\eps d\xi dx.
\end{split}
\end{equation}
Applying the Gronwall inequality to  \eqref{inequality on k order},
there exists constant $K>0$ such that,
$$\int_{\T^3}\int_{\R^3}|\xi|^kf^n_m(t,x,\xi) d\xi dx\leq Ke ^{Kt},$$
for any $m,n>0,$ and $k\geq 1.$ Here  $K$ depends on the initial data.
This estimate, together  with Lemma \ref{Lem2.3},  yields
\eqref{E4.10-}, \eqref{E4.11} and \eqref{E4.10}.

\ep

Similarly, we can show the following lemma:
\begin{lemma}\label{Lem2.4-1}
 Assume that $f_0\in L^\infty(\T^3\times\R^3)\cap L^1(\T^3\times\R^3)$ and $|\xi|^kf_0\in L^1(\T^3\times\R^3)$. If  $f_m^n\in L^\infty((0,T)\times\T^3\times\R^3)$, then
\beq\notag
M_kf_m^n(t)\leq C_{N,T}\left((M_kf_0)^{\frac{1}{3+k}}+(\|f_m^n\|_{L^\infty}+1)\|\tilde{u}\|_{L^p(0,T;L^{3+k}(\T^3))}\right)^{3+k},
\eeq
for all $0\leq t\leq T.$
\end{lemma}
\begin{proof}
Multiplying by $|\xi|^k$  on both sides of \eqref{E3.4-1}, one finds
\begin{equation*}
\begin{split}
&\frac{d}{d\tau}\int_{\T^3}\int_{\R^3}|\xi|^kf^n_m(t,x,\xi) d\xi dx- k\int_{\T^3}\int_{\R^3}(\tilde{u}-\xi)f_m^n|\xi|^{k-1}\,d\xi\,dx
\\&=-\lambda\int_{\T^3}\int_{\R^3}|\xi|^kf^n_m(t,x,\xi) d\xi dx+ \lambda \int_{\T^3}\int_{\R^3}\int_{\R^3}T(\xi,\xi')|\xi|^kf^{n-1}_m(t,x,\xi')\,d\xi' d\xi dx,\\
\end{split}
\end{equation*}
which yields

\begin{equation*}
\begin{split}
&\frac{d}{d\tau}\int_{\T^3}\int_{\R^3}|\xi|^kf^n_m(t,x,\xi) d\xi dx+ k\int_{\T^3}\int_{\R^3}|\xi|^kf_m^n\,d\xi\,dx+\lambda\int_{\T^3}\int_{\R^3}|\xi|^kf^n_m(t,x,\xi) d\xi dx
\\&\leq k\int_{\T^3}\int_{\R^3}|\tilde{u}||\xi|^{k-1}f_m^n\,d\xi\,dx+ \lambda \int_{\T^3}\int_{\R^3}\int_{\R^3}T(\xi,\xi')|\xi|^kf^{n-1}_m(t,x,\xi')\,d\xi' d\xi dx,\\
&\leq  k\int_{\T^3}\int_{\R^3}|\tilde{u}||\xi|^{k-1}f_m^n\,d\xi\,dx+ \lambda \int_{\T^3}\int_{\R^3}\int_{\R^3}|\xi|^kf^{n-1}_m(t,x,\xi)\,d\xi dx
\\
&\leq  k\int_{\T^3}\int_{\R^3}|\tilde{u}||\xi|^{k-1}f_m^n\,d\xi\,dx+ \lambda \int_{\T^3}\int_{\R^3}\int_{\R^3}|\xi|^kf^n_m(t,x,\xi)\,d\xi dx,
\end{split}
\end{equation*}
where  we have used that  $f_m^n$ is an increasing sequence with respect to $n$. Thus, we have
\begin{equation}
\label{E3.26}
\begin{split}
&\frac{d}{d\tau}\int_{\T^3}\int_{\R^3}|\xi|^kf^n_m(t,x,\xi) d\xi dx+ k\int_{\T^3}\int_{\R^3}|\xi|^kf_m^n\,d\xi\,dx\\
&\leq  k\int_{\T^3}\int_{\R^3}|\tilde{u}||\xi|^{k-1}f_m^n\,d\xi\,dx.
\end{split}
\end{equation}
Using H\"{o}lder's inequality,  the right-hand side of \eqref{E3.26} can be estimated as follows:
\beq\notag
k\int_{\T^3}\int_{\R^3}|\xi|^{k-1}f_m^{n}|\tilde{u}|d\xi dx\leq k\|\tilde{u}\|_{L^q(\T^3)}\|m_{k-1}f_m^{n}\|_{L^{q^\prime}(\T^3)},
\eeq
here $\frac{1}{q}+\frac{1}{q^\prime}=1$.  Let $R>0$ be fixed, then we have
\beq\notag
m_{k-1}f_m^n(t,x)\leq C\|f_m^n(t)\|_{L^\infty}R^{k+2}+\frac{1}{R}\int_{|\xi|>R} |\xi|^kf_m^nd\xi,
\eeq
taking $R=(\int_{\R^3} |\xi|^kf_m^nd\xi)^{\frac{1}{k+3}}$ and $q=k+3$,   we get
\beq\notag
k\int_{\T^3}\int_{\R^3}|\xi|^{k-1}f_m^n|\tilde{u}|d\xi dx\leq  Ck \|w(t)\|_{L^{k+3}(\T^3)}(\|f_m^n(t)\|_{L^\infty}+1)\left(\int_{\T^3}\int_{\R^3} |\xi|^kf_m^nd\xi dx\right)^{\frac{k+2}{k+3}},
\eeq
substituting this into \eqref{E3.26},  we complete the proof of Lemma \ref{Lem2.4-1}.
\end{proof}

Thus, we are ready to show the following existence of weak solutions of approximation \eqref{E3.3-1}.
\begin{lemma}
For any $T>0$ and fixed $m>0,$ there exists a weak solution of the  approximation \eqref{E3.3-1} in the following sense:
\begin{equation}\label{E3.17}
\begin{split}
&\int_{\mathbb{T}^3}\int_{\mathbb{R}^3}f_m(t)\phi(t,x,\xi)d\xi dx-\int_0^t\int_{\mathbb{T}^3}\int_{\mathbb{R}^3}f_m(\phi_t+\xi\cdot \nabla_x \phi +(\tilde{u}-\xi)\cdot \nabla_{\xi}\phi)d\xi dx ds\\
=&\int_{\mathbb{T}^3}\int_{\mathbb{R}^3}f^\epi_0\phi(0,x,\xi)d\xi dx
+\lambda\int_0^t\int_{\mathbb{T}^3}\int_{\mathbb{R}^3}f_m\phi d\xi dx ds\\
&
-\lambda\int_0^t\int_{\mathbb{T}^3}\int_{\mathbb{R}^3}\int_{\mathbb{R}^3}T(\xi,\xi^\prime)f_m(t,x,\xi^\prime)\phi d\xi^\prime d\xi dx ds.
\end{split}
\end{equation}
In particular, the solution satisfies the following bounds:
\beq\label{E3.27}
\|f_m(t,x,\xi)\|_{ L^\infty(0,T;L^p(\T^3\times\R^3))}\leq C(\lambda,T,\|f^\epi_0\|_{L^\infty\cap L^1(\T^3\times \R^3)}),\ for\ any\ p\geq 1,
\eeq
\beq\label{E3.28}
\left\|\int_{\mathbb{R}^3}T(\xi,\xi^\prime)f_m(t,x,\xi^\prime)d\xi^\prime\right\|_{\ L^\infty(0,T;L^p(\T^3\times\R^3))}\leq C(\lambda,T,\|f^\epi_0\|_{L^\infty\cap L^1(\T^3\times \R^3)}),\ for\ any\ p\geq 1,
\eeq
\beq\label{E3.29}
\left\|\int_{\R^3} f_m d \xi\right\|_{L^\infty(0,T;L^2(\T^3))}\leq C(\lambda,T,M_3f_0^\eps,\|f^\epi_0\|_{L^\infty},\tilde{u}),
\eeq
\beq\label{E3.30-}
\left\|\int_{\R^3}\xi f_m d \xi\right\|_{L^\infty(0,T;L^\frac{3}{2}(\T^3))}\leq C(\lambda,T,M_3 f_0^\eps,\|f^\epi_0\|_{L^\infty},\tilde{u}),
\eeq
and
\beq\label{E3.31}
\left\|\int_{\R^3}|\xi|^2 f_m d \xi\right\|_{L^\infty(0,T;L^\frac{6}{5}(\T^3))}\leq C(\lambda,T,M_3 f_0^\eps,\|f^\epi_0\|_{L^\infty},\tilde{u}).
\eeq

\end{lemma}
\begin{proof}
First,   from  \eqref{E4.10-}, \eqref{E4.10} and \eqref{E4.11}, we deduce the following convergence results
\beq\label{E4.17-}
\int_{\R^3} f^n_m d \xi \rightharpoonup  \int_{\R^3} f_m d\xi \ weakly(*) \ in\ L^\infty(0,T;L^2(\T^3)),
\eeq
\beq\label{E4.17}
\int_{\R^3}\xi f^n_m d \xi \rightharpoonup  \int_{\R^3}\xi f_m d\xi \ weakly(*) \ in\ L^\infty(0,T;L^\frac{3}{2}(\T^3)),
\eeq
and
\beq\label{E4.18}
\int_{\R^3}|\xi|^2 f^n_m d \xi \rightharpoonup  \int_{\R^3}|\xi|^2 f_m d\xi \ weakly(*) \ in\ L^\infty(0,T;L^\frac{6}{5}(\T^3)),
\eeq
as $n\rightarrow \infty.$

In order to pass to the limits as $n\to \infty,$ we investigate  the convergence of the operator $\int_{\mathbb{R}^3}T(\xi,\xi^\prime)f^{n}_m(t,x,\xi^\prime)d\xi'$ in a  suitable space.
To this end, for any function $\psi(t,x)$ in $L^2(0,T;L^2(\R^3))$,  we find
\begin{equation}
\begin{split}
&\int_{\T^3}\int_{\R^3}\left(\int_{\mathbb{R}^3}T(\xi,\xi^\prime)f^{n-1}_m(t,x,\xi^\prime)d\xi^\prime-\int_{\mathbb{R}^3}T(\xi,\xi^\prime)f_m(t,x,\xi^\prime)d\xi^\prime\right)\psi(t,x)d\xi dx\\
=& \int_{\T^3}\int_{\R^3}\int_{\R^3}T(\xi,\xi^\prime)(f^{n-1}_m(t,x,\xi^\prime)-f_m(t,x,\xi^\prime))\psi(t,x)d\xi^\prime d\xi dx\\
=&  \int_{\T^3}\int_{\R^3}(f^{n-1}_m(t,x,\xi)-f_m(t,x,\xi))\psi(t,x) d\xi dx
\rightarrow 0
\end{split}
\end{equation}
as $n\rightarrow \infty$, thanks to \eqref{E4.17-}.    By Lemma \ref{Lem4.2}, this yields
\beq\label{E4.12}
\int_{\mathbb{R}^3}T(\xi,\xi^\prime)f^{n}_m(t,x,\xi^\prime)d\xi^\prime \rightharpoonup \int_{\mathbb{R}^3}T(\xi,\xi^\prime)f_m(t,x,\xi^\prime)d\xi^\prime \ weakly(*)\ in\  L^\infty(0,T;L^2(\T^3)),\\
\eeq
 for any $q>1$.

\bigbreak
Since $f_m^n$ given by \eqref{E3.4} is a smooth solution of problem \eqref{E3.4-1}, it satisfies the following weak formulation
\begin{equation}\label{E3.15}
\begin{split}
&\int_{\mathbb{T}^3}\int_{\mathbb{R}^3}f^n_m(t,x,\xi)\phi(t,x,\xi)d\xi dx-\int_0^t\int_{\mathbb{T}^3}\int_{\mathbb{R}^3}f^n_m(\phi_t+\xi\cdot \nabla_x \phi +(\tilde{u}-\xi)\cdot \nabla_{\xi}\phi)d\xi dx ds\\
=&\int_{\mathbb{T}^3}\int_{\mathbb{R}^3}f^\epi_0\phi(0,x,\xi)d\xi dx
+\lambda\int_0^t\int_{\mathbb{T}^3}\int_{\mathbb{R}^3}f^n_m\phi d\xi dx ds\\
&
-\lambda\int_0^t\int_{\mathbb{T}^3}\int_{\mathbb{R}^3}\int_{\mathbb{R}^3}T(\xi,\xi^\prime)f^{n-1}_m(t,x,\xi^\prime)\phi d\xi^\prime d\xi dx ds,
\end{split}
\end{equation}
for any test function $\phi\in C^\infty([0,\infty)\times\mathbb{T}^3\times\mathbb{R}^3)$.\\
  Letting $n$ tend to infinity in  \eqref{E3.15}, and using \eqref{E4.12}, \eqref{E4.17-} and \eqref{E4.17}, one obtains
\begin{equation*}
\begin{split}
&\int_{\mathbb{T}^3}\int_{\mathbb{R}^3}f_m(t)\phi(t,x,\xi)d\xi dx-\int_0^t\int_{\mathbb{T}^3}\int_{\mathbb{R}^3}f_m(\phi_t+\xi\cdot \nabla_x \phi +(\tilde{u}-\xi)\cdot \nabla_{\xi}\phi)d\xi dx ds\\
=&\int_{\mathbb{T}^3}\int_{\mathbb{R}^3}f^\epi_0\phi(0,x,\xi)d\xi dx
+\lambda\int_0^t\int_{\mathbb{T}^3}\int_{\mathbb{R}^3}f_m\phi d\xi dx ds\\
&
-\lambda\int_0^t\int_{\mathbb{T}^3}\int_{\mathbb{R}^3}\int_{\mathbb{R}^3}T(\xi,\xi^\prime)f_m(t,x,\xi^\prime)\phi d\xi^\prime d\xi dx ds.
\end{split}
\end{equation*}

\noindent  Here we should  remark that the solution $f_m$ satisfies the following bounds
\beq\label{E3.27!}
\|f_m\|_{ L^\infty(0,T;L^p(\T^3\times\R^3))}\leq C,\ for\ any\ p\geq 1,
\eeq
\beq\label{E3.28!}
\left\|\int_{\mathbb{R}^3}T(\xi,\xi^\prime)f_m(d\xi^\prime\right\|_{\ L^\infty(0,T;L^p(\T^3\times\R^3))}\leq C,\ for\ any\ p\geq 1,
\eeq
\beq\label{E3.29!}
\left\|\int_{\R^3} f_m d \xi\right\|_{L^\infty(0,T;L^2(\T^3))}\leq C,
\eeq
\beq\label{E3.30-!}
\left\|\int_{\R^3}\xi f_m d \xi\right\|_{L^\infty(0,T;L^\frac{3}{2}(\T^3))}\leq C,
\eeq
and
\beq\label{E3.31!}
\left\|\int_{\R^3}|\xi|^2 f_m d \xi\right\|_{L^\infty(0,T;L^\frac{6}{5}(\T^3))}\leq C,
\eeq
where $C$ are positive constants depend only on $\lambda,T,M_3 f_0^\eps$ and $\|f^\epi_0\|_{L^\infty\cap L^1(\T^3\times \R^3)}$.
\end{proof}
\bigbreak
\subsection{ The Navier-Stokes part}
\noindent In this subsection, we shall study the solution of the Navier-Stokes part, and the energy inequality for the whole approximation \eqref{E3.1}.  First, we are able to view the right-hand side of  \eqref{E3.1}$_2$ as a external force of the Navier-Stokes equations.   From \eqref{E3.29} and \eqref{E3.30-}, the right-hand side of \eqref{E3.1}$_2$ can be estimated as follows:
 \beq\label{E3.33--}
\begin{split}
& \left\|-\int_{\mathbb{R}^3}f_m(\tilde{u}-\xi)d\xi\right\|_{L^\infty(0,T;L^{\frac{3}{2}}(\T^3))}\\
&=\left\| -\tilde{u}\int_{\mathbb{R}^3}f_md\xi +\int_{\mathbb{R}^3}\xi f_md\xi\right\|_{L^\infty(0,T;L^{\frac{3}{2}}(\T^3))}\\
& \leq C(m,\lambda,T,M_3 f^\epi_0,\|f^\epi_0\|_{L^\infty\cap L^1(\T^3\times \R^3)},\|\tilde{u}\|_{L^\infty(0,T;L^2(\T^3))}).
\end{split}
\eeq
Thus, we can apply the classical theory of Navier-Stokes equations to solve  \eqref{E3.1}$_2$ when $f_{m}$ is a solution of \eqref{E3.3-1}.
\bigbreak

\noindent
 Next we consider  the following weak formulation of the Navier-Stokes equation \eqref{E3.1}$_2$:
\begin{equation}\label{E3.14}
\begin{split}
& \int_{\mathbb{T}^3}[\partial_tu_m\cdot \varphi+(\tilde{u}\cdot \nabla) u_m\cdot \varphi+\nabla u_m: \nabla \varphi]dx\\
=&-\int_{\mathbb{T}^3}\left(\int_{\mathbb{R}^3}f_m(\tilde{u}-\xi)d\xi\right) \cdot \varphi dx,
\end{split}
\end{equation}
where  $\varphi\in X_m$. Since  $X_m$ is a finite-dimensional space, we can write $u_m$ as follows
\beq\notag
u_m=\sum\limits_{i=1}^{m}\alpha_{im}(t) \varphi_i.
\eeq
By the standard Galerkin method, the approximation \eqref{E3.1}$_2$  yields the following ODE:
\beq\label{E3.34}
\begin{split}
\frac{d}{dt}\alpha_{im}(t)=&-\int_{\T^3}(\tilde{u}\cdot \nabla \varphi_j)\cdot\varphi_idx   \alpha_{jm}(t)-\int_{\T^3} \nabla \varphi_i:\nabla\varphi_j dx \alpha_{jm}(t)\\
-& \int_{\T^3}\left(\int_{\mathbb{R}^3}f_m(\tilde{u}-\xi)d\xi \right)\cdot \varphi_idx.
\end{split}
\eeq
From \eqref{E3.33--}, then by using the classical ODE theory, there exists a unique solution  $\alpha_{im}(t)$$(i=1,\cdots,m)$ of \eqref{E3.34} for any $0\leq t\leq  T_m$, where $0<T_m\leq T.$ This gives us a unique solution $u_m\in Y_m$ of the weak formulation \eqref{E3.14} for any $0\leq t\leq T_m$.

 Next, we can derive the energy inequality for $u_m$. Indeed, taking $\varphi_i$ for \eqref{E3.14} and multiplying the equation by $\alpha_{im}$, then summing  the resulting equality from $i=1$ to $m$,  we have
\beq\notag
\begin{split}
&\frac{1}{2}\frac{d}{dt}\int_{\T^3}|u_m|^2dx+\int_{\T^3}|\nabla u_m|^2dx\\
=& -\int_{\T^3}\int_{\R^3}f_m(\tilde{u}-\xi)\cdot u_md\xi dx\\
\leq & \|u_m\|_{L^\infty(\T^3)}\cdot\left\|\int_{\R^3}f_m(\tilde{u}-\xi) d\xi\right\|_{L^1(\T^3)}\\
\leq & \frac{1}{2}\int_{\T^3}|\nabla u_m|^2dx+ C(m)\left\|\int_{\R^3}f_m(\tilde{u}-\xi) d\xi\right\|_{L^{\frac{3}{2}}(\T^3)}^2,
\end{split}
\eeq
where we have used the  fact that all norms in $X_m$ are equivalent. This yields
\beq\label{E3.35}
\begin{split}
&\int_{\T^3}|u_m|^2dx+\int_0^t\int_{\T^3}|\nabla u_m|^2dxds\leq \int_{\T^3}|u_0^\epi|^2dx
\\&\quad\quad\quad\quad\quad\quad\quad\quad\quad+C(m)\int_0^t \left\|\int_{\R^3}f_m(\tilde{u}-\xi) d\xi\right\|_{L^{\frac{3}{2}}(\T^3)}^2ds.
\end{split}
\eeq
 With the help of \eqref{E3.33--}, this gives us the following estimates
 \begin{equation}
 \begin{split}&
 \label{local estimate for NS equations}\|u_m\|_{L^2(\T^3)}\leq M<\infty
 \quad\text{ for any  } t\in[0,T_m],\\
 & \|\nabla u_m\|_{L^2(0,T;L^2(\T^3))}\leq M.
 \end{split}
 \end{equation}

Now we define a convex set
\beq\notag
\mathbf{A}:=\{\tilde{u}\in C([0,T_m];X_m):  \sup\limits_{0\leq t\leq T_m}\|\tilde{u}\|_{L^2(\T^3)}\leq M,\;\;\Dv\tilde{u}=0\},
\eeq
and a map $S:\mathbf{A}\to \mathbf{A}$ such that $u_m:=S(\tilde{u})$.
We shall apply   Schauder's fixed point theorem to show that the operator $S$ has a fixed point in the following lemma.

\begin{lemma}\label{lemma3.4}
The operator $S$ has a fixed point in $\mathbf{A}$, that is, there is  a point $u_m\in \mathbf{A}$ such that $Su_m=u_m=\tilde{u}$.
\end{lemma}
\begin{proof} Thanks to \eqref{local estimate for NS equations}, $u_m$ is bounded in the set $\mathbf{A}.$
Meanwhile, taking $\varphi$ in \eqref{E3.14} to be $\varphi_i$, multiplying the equation by $\alpha_{im}^\prime(t)$,  then summing the resulting equality   from $i=1$ to $m$,    we have
\beq\notag
\begin{split}
&\int_{\T^3}|\partial_t u_m|^2dx +\frac{1}{2}\frac{d}{dt}\int_{\T^3} |\nabla u_m|^2 dx =-\int_{\mathbb{T}^3}\left(\int_{\mathbb{R}^3}f_m(\tilde{u}-\xi)d\xi\right) \cdot \partial_t u_m dx\\
&\quad\quad\quad \quad\quad\quad \quad\quad\quad \quad\quad\quad \quad\quad\quad -\int_{\T^3} ((\tilde{u}\cdot \nabla)u_m)\cdot \partial_t u_m dx\\
&\leq \|\partial_t u_m \|_{L^\infty(\T^3)}\cdot\|\tilde{u}\|_{L^2(\T^3)}\cdot\|\nabla u_m\|_{L^2(\T^3)}
+\|\partial_t u_m \|_{L^\infty(\T^3)}\cdot \left\|\int_{\R^3}f_m(\tilde{u}-\xi) d\xi\right\|_{L^{1}(\T^3)}\\
&\leq \frac{1}{2}\int_{\T^3}|\partial_t u_m|^2dx+C(m) \left\|\int_{\R^3}f_m(\tilde{u}-\xi) d\xi\right\|_{L^{1}(\T^3)}^2 +C(m)\|\tilde{u}\|_{L^2(\T^3)}^2\|\nabla u_m\|_{L^2(\T^3)}^2.
\end{split}
\eeq
Note that all norms in $X_m$ are equivalent to each other,  so by  \eqref{E3.33--} and \eqref{local estimate for NS equations}  one obtains that
\beq
\begin{split}
\label{local estimate for Aubin's lemma}
\int_0^t\int_{\T^3}|\partial_t u_m|^2dxds \leq C(m,\lambda,T, M_3 f^\epi_0,\|f^\epi_0\|_{L^\infty\cap L^1(\T^3\times \R^3)},\|\tilde{u}\|_{L^\infty(0,T;L^2(\T^3))}),
\end{split}
\eeq
for any $t\in[0,T_m].$
  Thanks to \eqref{local estimate for NS equations}$_2$ and \eqref{local estimate for Aubin's lemma}, the Aubin-Lions  Lemma,  $u_m=S(\tilde{u})$ is compact in $\mathbf{A}$.  On the other hand, it is easy to verify that $S$ is sequentially continuous, see \cite{Legar,Yu2} for the details.  Thus  Schauder's  fixed point   theorem gives that  $S$ has a fixed point $u_m$ in $\mathbf{A}.$
\end{proof}

%With the above lemma at hand, we are ready to establish the energy estimate on $(u,f)$ for the whole system, which can ensure us to extend $T_m$ to $T$.
%Here we derive the following energy estimate in Lemma.

%\begin{lemma}\label{lemma3.4--}
%The operator $S$ has a fixed point in $\mathbf{A}$, that is, there is  a point $u_m\in \mathbf{A}$ such that $Su_m=u_m=\tilde{u}$.
%\end{lemma}
%\begin{proof} By \eqref{E3.35}, $u_m$ is bounded in $C([0,T_m],L^2(\T^3))\cap L^2(0,T;H^1(\T^2))$, thus it is in the set $\mathbf{A}.$
%From  \eqref{E3.14}, we deduce that $\partial_t u_m$ is bounded in $L^2(0,T;W^{-1,3}(\T^3)).$
%Then by Aubin-Lions  lemma, we have $u_m=S(\tilde{u})$ is compact in $Y_m=C([0,T];X_m)$.  On the other hand, it is easy to verify that $S$ is sequentially continuous, see \cite{Legar,Yu2} for the details.  So Schauder's  fixed point   theorem will give  $S$ has a fixed point $u_m$ in $Y_m.$
%\end{proof}

We have shown that there exists   a pair $(f_m,u_m)$ on a short time interval $[0,T_m].$ In an effort to extend $T_m$ to $T$, we rely on the following  uniform bounds on $u_m$.
\begin{lemma}\label{lemma3.5}
For any $t\in [0,T_m],$  we have
\begin{equation}\label{E3.37-}
\begin{split}
&\frac{1}{2}\int_{\mathbb{T}^3}|u_m|^2dx+\int_{\mathbb{T}^3}\int_{\mathbb{R}^3}f_m(1+\frac{1}{2}|\xi|^2)d\xi dx+\int_0^t \int_{\mathbb{T}^3}|\nabla_x u_m|^2dx ds\\
&+\int_0^t \int_{\mathbb{T}^3}\int_{\mathbb{R}^3}f_m|u_m-\xi|^2d\xi dx  ds\\
\leq &\frac{1}{2}\int_{\mathbb{T}^3}|u_0^\epi|^2dx+\int_{\mathbb{T}^3}\int_{\mathbb{R}^3}f_0^\epi(1+\frac{1}{2}|\xi|^2)d\xi dx.
\end{split}
\end{equation}
\end{lemma}
\begin{proof}
Taking $\varphi$ in \eqref{E3.14} to be $\varphi_i$ and multiplying the equation by $\alpha_{im}$,   summing   from $i=1$ to $m$,  we have
\begin{equation}
\label{ssss}
\begin{split}
& \frac{1}{2}\int_{\T^3}|u_m|^2dx+\int_0^t\int_{\T^3}|\nabla u_m|^2dxds\\
=& \frac{1}{2}\int_{\T^3}|u_0^\epi|^2dx-\int_0^t\int_{\T^3}\int_{\R^3}f_m(\tilde{u}-\xi)\cdot u_md\xi dxds.\\
\end{split}
\end{equation}
Note that $\tilde{u}=u_m$ from Lemma \ref{lemma3.4}.  We are able to use $u_m$ to replace $\tilde{u}$ in the energy equality of the Navier-Stokes part \eqref{ssss} on a short  time $[0,T_m]$,  thus
\begin{equation}
\begin{split}
\label{E3.38}
& \frac{1}{2}\int_{\T^3}|u_m|^2dx+\int_0^t\int_{\T^3}|\nabla u_m|^2dxds\\
=& \frac{1}{2}\int_{\T^3}|u_0^\epi|^2dx-\int_0^t\int_{\T^3}\int_{\R^3}f_m(u_m-\xi)\cdot u_md\xi dxds.
\end{split}
\end{equation}

   Note that, $f_m^n$ given by \eqref{E3.4} is a smooth solution to the problem \eqref{E3.4-1}. We can multiply $1+\frac{1}{2}|\xi|^2$ on both sides of \eqref{E3.4-1} and integrate to deduce the  following equality
\begin{equation}
\begin{split}
\label{energy for kinetic part}
& \int_{\mathbb{T}^3}\int_{\mathbb{R}^3}f_m^n(1+\frac{1}{2}|\xi|^2)d\xi dx+\int_0^t\int_{\mathbb{T}^3}\int_{\mathbb{R}^3}f_m^n|\tilde{u}-\xi|^2d\xi dxds\\
=&\int_{\mathbb{T}^3}\int_{\mathbb{R}^3}f_0^\epi(1+\frac{1}{2}|\xi|^2)d\xi dx+\int_0^t\int_{\mathbb{T}^3}\int_{\mathbb{R}^3}f_m^n(\tilde{u}-\xi)\cdot \tilde{u}d\xi dxds-\lambda\int_0^t\int_{\mathbb{T}^3}\int_{\mathbb{R}^3}f_m^n(1+\frac{1}{2}|\xi|^2)d\xi dxds\\
&+ \lambda\int_0^t\int_{\mathbb{T}^3}\int_{\mathbb{R}^3}\int_{\mathbb{R}^3}T(\xi,\xi^\prime) f_m^{n-1}(t,x,\xi^\prime)d\xi^\prime(1+\frac{1}{2}|\xi|^2)d\xi dxds\\
=&\int_{\mathbb{T}^3}\int_{\mathbb{R}^3}f_0^\epi(1+\frac{1}{2}|\xi|^2)d\xi dx+\int_0^t\int_{\mathbb{T}^3}\int_{\mathbb{R}^3}f_m^n(\tilde{u}-\xi)\cdot \tilde{u}d\xi dxds-\lambda\int_0^t\int_{\mathbb{T}^3}\int_{\mathbb{R}^3}f_m^n(1+\frac{1}{2}|\xi|^2)d\xi dxds\\
&+ \lambda\int_0^t\int_{\mathbb{T}^3}\int_{\mathbb{R}^3} f_m^{n-1}(t,x,\xi)(1+\frac{1}{2}|\xi|^2)d\xi dxds,
\end{split}
\end{equation}
 where we have used the Fubini's theorem, \eqref{E1.3} and \eqref{E1.4} in the last equality of \eqref{energy for kinetic part}. Next,  by the convergence \eqref{E4.17-}-\eqref{E4.12}, we are able to pass to the limits in \eqref{energy for kinetic part} as $n\to\infty$. In fact, we can recover the following inequality
\begin{equation}\label{TTTT}
\begin{split}
& \int_{\mathbb{T}^3}\int_{\mathbb{R}^3}f_m(1+\frac{1}{2}|\xi|^2)d\xi dx+\int_0^t\int_{\mathbb{T}^3}\int_{\mathbb{R}^3}f_m|\tilde{u}-\xi|^2d\xi dxds\\
\leq &\int_{\mathbb{T}^3}\int_{\mathbb{R}^3}f_0^\epi(1+\frac{1}{2}|\xi|^2)d\xi dx+\int_0^t\int_{\mathbb{T}^3}\int_{\mathbb{R}^3}f_m(\tilde{u}-\xi)\cdot \tilde{u}d\xi dxds.
\end{split}
\end{equation}
Applying Lemma \ref{lemma3.4} to \eqref{TTTT}, one obtains
 \begin{equation}\label{E3.39}
\begin{split}
& \int_{\mathbb{T}^3}\int_{\mathbb{R}^3}f_m(1+\frac{1}{2}|\xi|^2)d\xi dx+\int_0^t\int_{\mathbb{T}^3}\int_{\mathbb{R}^3}f_m|u_m-\xi|^2d\xi dxds\\
\leq &\int_{\mathbb{T}^3}\int_{\mathbb{R}^3}f_0^\epi(1+\frac{1}{2}|\xi|^2)d\xi dx+\int_0^t\int_{\mathbb{T}^3}\int_{\mathbb{R}^3}f_m(u_m-\xi)\cdot u_m d\xi dxds.
\end{split}
\end{equation}
  Combining \eqref{E3.38} and \eqref{E3.39}, we have \eqref{E3.37-} for any $t\in[0,T_m].$

\end{proof}

Note that all norms in $X_m$ are equivalent, Lemma \ref{lemma3.5} yields a uniform estimate on $u_m$ as follows $$\sup\limits_{0\leq t\leq T_m}\|u_m\|_{X_m}\leq C(m)<\infty.$$
 It allows us to have $T_m=T.$ Thus, the solution $(u_m,f_m)$ exists on $[0,T]$. Hence, we have the following result on the global existence of weak solutions of the approximation system:
\begin{proposition}\label{Pro3.1}
For any  $T>0$, under the assumptions of Theorem \ref{Thm 1}, there exists a  weak solution $(u_m, f_m)$  of the following problem
\begin{equation}\label{E3.18}
\left\{\begin{array}{l}
\partial_t f_m
  +
  \xi\cdot \nabla_x f_m+{\rm div}_\xi((u_m-\xi)f_\epsilon)=-\lambda f_m+\lambda\int_{\mathbb{R}^3}T(\xi,\xi^\prime)f_m(t,x,\xi^\prime)d\xi^\prime,\\[2mm]
\partial_t u_m
 +(u_m\cdot \nabla_x)u_m +\nabla_x P_m- \Delta_x u_m=-\int_{\mathbb{R}^3}(u_m-\xi)f_m  d\xi,    \\[2mm]
 {\rm div}u_m=0,  \ \ \ \   \qquad x\in \mathbb{T}^3,\ \xi\in \mathbb{R}^3, \  t\in(0,T), \\[2mm]
 \end{array}
        \right.
\end{equation}
with initial data
\begin{equation}\label{E3.19-}
(f_m,u_m)|_{t=0}=(f_0^\epsilon(x,\xi),u_0^\epsilon(x)), \ \ \  x\in \mathbb{T}^3,\ \xi\in \mathbb{R}^3.
\end{equation}
Moreover, for any $0\leq t\le T$, the weak solution   $(u_m, f_m)$ satisfies the following energy inequality
\begin{equation}\label{E3.19}
\begin{split}
&\frac{1}{2}\int_{\mathbb{T}^3}|u_m|^2dx+\int_{\mathbb{T}^3}\int_{\mathbb{R}^3}f_m(1+\frac{1}{2}|\xi|^2)d\xi dx+\int_0^T \int_{\mathbb{T}^3}|\nabla_x u_m|^2dx ds\\
&+\int_0^T \int_{\mathbb{T}^3}\int_{\mathbb{R}^3}f_m|u_m-\xi|^2d\xi dx  ds\\
\leq & \frac{1}{2}\int_{\mathbb{T}^3}|u_0^\epi|^2dx+\int_{\mathbb{T}^3}\int_{\mathbb{R}^3}f_0^\epi(1+\frac{1}{2}|\xi|^2)d\xi dx,
\end{split}
\end{equation}
 and
\beq\label{E3.20}
\|f_m\|_{ L^\infty(0,T;L^p(\R^3\times\T^3))}\leq C,\ for\ any\ 1\leq p\leq \infty,
\eeq
where $C>0$ depends only on the initial data, $\lambda$ and $T$.
\end{proposition}

\bigbreak
\section{Recover the weak solutions}
The main goal of this section is to recover the weak solutions of problem \eqref{E1.1}-\eqref{E1.2} by passing to the limit in the sequence  $(u_m,f_m)$ which was constructed in
Proposition \ref{Pro3.1}.
In particular,
we shall pass to the limits as $m$ goes to infinity and $\epi$ tends to zero, and show that
the limit function is a weak solution of problem \eqref{E1.1}-\eqref{E1.2}.
In the following, we will investigate the weak limit
with respect to $m$ in Step 1 and pass to the limit with respect to $\epi$ in
Step 2.

\bigbreak

\textbf{Step 1. Passing to the limit as $m\rightarrow \infty$.}

In this step, we keep $\epi>0$ fixed, deducing from Proposition \ref{Pro3.1}, we have the following estimates independent of $m$:
\beq\label{E4.1}
\|f_m\|_{ L^\infty(0,T;L^p(\T^3\times\R^3))}\leq C,\text{ for any}\ 1\leq p\leq  \infty,
\eeq
\beq\label{E4.4}
M_2f_m(t)\leq C,\text{ for any}\ 0\leq t\leq T,
\eeq
\beq\label{E4.2-}
\|u_m\|_{ L^\infty(0,T;L^2(\T^3))}\leq C,
\eeq
\beq\label{E4.3}
\|u_m\|_{ L^2(0,T;H^1(\T^3))}\leq C.
\eeq
Meanwhile, as the same in Section \ref{S3}, the solution satisfies the following estimate:
\beq\label{E3.28--}
\left\|\int_{\mathbb{R}^3}T(\xi,\xi^\prime)f_m(t,x,\xi^\prime)d\xi^\prime\right\|_{\ L^\infty(0,T;L^p(\T^3\times\R^3))}\leq C,\text{ for any } 1\leq p\leq \infty.
\eeq
Furthermore, by \eqref{E4.1}-\eqref{E4.3} and Lemmas \ref{Lem2.3} and Lemma \ref{Lem2.4-1}, any solution satisfies the following uniform bounds
\beq\label{E4.39--}
M_3f_m(t)\leq C,\text{ for any}\ 0\leq t\leq T,
\eeq
\beq\label{E4.39-}
\left\|\int_{\R^3} f_m d \xi\right\|_{L^\infty(0,T;L^2(\T^3))}\leq C,
\eeq
\beq\label{E4.40-}
\left\|\int_{\R^3}\xi f_m d \xi\right\|_{L^\infty(0,T;L^\frac{3}{2}(\T^3))}\leq C,
\eeq
and
\beq\label{E4.41-}
\left\|\int_{\R^3}|\xi|^2 f_m d \xi\right\|_{L^\infty(0,T;L^\frac{6}{5}(\T^3))}\leq C.
\eeq

\bigbreak

With the above estimates \eqref{E4.1}-\eqref{E4.41-}, we are ready to investigate the limit  as $m$ goes to infinity. To this end, we shall rely on the Aubin-Lions lemma for the Navier-Stokes part and the $L^p$ average velocity lemma for the kinetic part.
Using the same arguments as that in \cite{Boudin}, we prove that $\partial_t u_m$ is bounded in ${L^{\frac{4}{3}}(0,T;H^{-1}(\mathbb{T}^3))}$. For   completeness, we give the proof as follows.
\begin{lemma}\label{Lem4.4}
For any $m$, it holds that
\beq\label{E4.29}
\|\partial_t u_m\|_{L^{\frac{4}{3}}(0,T;H^{-1}(\mathbb{T}^3))}\leq C.
\eeq
\end{lemma}
\bp Note that $$\partial_t u_m=-
 (u_m\cdot \nabla_x)u_m-\nabla_x P_m+\Delta_x u_m+\int_{\mathbb{R}^3}(u_m-\xi)f_m  d\xi. $$
 We control the first term
\beq\label{E4.30}
\int_0^{T}\int_{\mathbb{T}^3} (u_m\cdot \nabla) u_m\cdot\varphi dxds=-\int_0^{T}\int_{\mathbb{T}^3} (u_m\cdot \nabla)\varphi\cdot u_m dxds.
\eeq
By \eqref{E4.2-}-\eqref{E4.3} and an interpolation inequality, one obtains
\beq\label{E4.31}
\|u_m\|_{L^{\frac{8}{3}}(0,T;L^4(\T^3))}\leq C.
\eeq
From \eqref{E4.30} and \eqref{E4.31}, we deduce that
\beq
\left|\int_0^{T}\int_{\mathbb{T}^3}( u_m\cdot \nabla) u_m\cdot\varphi dxds\right|\leq C \|\nabla \varphi\|_{L^{4}(0,T;L^2(\T^3))}.
\eeq
The second term vanishes if $\Dv\varphi=0,$ and the third term is controlled by
$$\int_0^T\int_{\mathbb{T}^3}\Delta u_m \varphi\,dx\,dt=-\int_0^T\int_{\mathbb{T}^3}\nabla u_m :\nabla\varphi\,dx\,dt,$$
which is bounded by
$C\|\nabla\varphi\|_{L^2(0,T;L^2(\T^3))}.$

Next, for the last term, by H\"{o}lder inequality, we have
\beq
\begin{split}
&\left|\int_0^{T}\int_{\mathbb{T}^3}\int_{\mathbb{R}^3}f_m(u_m-\xi)\cdot \varphi d\xi dx ds\right|\\
\leq & C\|u_m\|_{L^{2}(0,T;L^6(\T^3))}\|\varphi\|_{L^{2}(0,T;L^6(\T^3))}\|m_0f_m\|_{L^{\infty}(0,T;L^{\frac{3}{2}}(\T^3))}\\
& +  C\|\varphi\|_{L^{2}(0,T;L^5(\T^3))}\|m_1f_m\|_{L^{2}(0,T;L^{\frac{5}{4}}(\T^3))},
\end{split}
\eeq
which implies that
$
\int_{ \mathbb{R}^3}f_m(u_m-\xi)d\xi
$
is uniformly bounded in  $L^{2}(0,T;H^{-1}(\T^3))$.  This completes the proof of Lemma  \ref{Lem4.4}.

\ep

\noindent  With \eqref{E4.3} and Lemma \ref{Lem4.4}, the Aubin-Lions Lemma yields
\begin{equation}
\label{strong convergence of velocity}
u_m\to u\quad\text{ strongly in } L^{\infty}(0,T;L^r(\mathbb{T}^3))
\end{equation}
for any $1<r\leq 6.$

 By \eqref{E4.1}, \eqref{E4.2-}-\eqref{E4.3},  as $m\rightarrow \infty,$ we have
\beq\label{E3.37}
\begin{split}
& f_m \rightharpoonup f, \ weakly(*)\ in\  L^\infty(0,T;L^q(\T^3\times\R^3),\ for\ any\ q>1;\\
& u_m \rightharpoonup u, \ \ weakly(*)\ in\   L^\infty(0,T;L^2(\T^3))\cap  L^2(0,T;H^1(\T^3)).
\end{split}
\eeq
Thanks to \eqref{E3.28--} and \eqref{E3.37}$_1$, we employ the same arguments as in \eqref{E4.12}, to have,  for any $q>1$, as $m\rightarrow \infty, $
\beq\label{E4.43-}
\int_{\mathbb{R}^3}T(\xi,\xi^\prime)f_m(t,x,\xi^\prime)d\xi^\prime \rightharpoonup \int_{\mathbb{R}^3}T(\xi,\xi^\prime)f(t,x,\xi^\prime)d\xi^\prime, \ weakly(*)\ in\  L^\infty(0,T;L^q(\T^3\times\R^3)).\\
\eeq

\noindent Furthermore,  by \eqref{E4.39-}-\eqref{E4.41-} and \eqref{E3.37}$_1$,  we have
\beq \label{E4.52}
\int_{\R^3}f_m d \xi \rightharpoonup  \int_{\R^3} f d\xi,\ weakly(*) \ in\ L^\infty(0,T;L^2(\T^3)),
\eeq
\beq\label{E4.53}
\int_{\R^3}\xi f_m d \xi \rightharpoonup  \int_{\R^3}\xi fd\xi,\ weakly(*) \ in\ L^\infty(0,T;L^\frac{3}{2}(\T^3)),
\eeq
and
\beq\label{E4.54}
\int_{\R^3}|\xi|^2 f_m d  \xi \rightharpoonup  \int_{\R^3}|\xi|^2 fd\xi,\ weakly(*) \ in\ L^\infty(0,T;L^\frac{6}{5}(\T^3)),
\eeq
as $m\rightarrow \infty.$

  \noindent Thus, we are ready to investigate the limits as $m$ goes to infinity. In particular, with \eqref{strong convergence of velocity}-\eqref{E3.37},  as $m\to\infty,$ from the weak formulation:
\begin{equation}\label{E4.37}
\begin{split}
&\int_0^{T}\int_{\mathbb{T}^3}(-u_m\cdot \partial_t\varphi+(u_m\cdot \nabla) u_m\cdot \varphi+\nabla u_m: \nabla \varphi)dxds\\
&=-\int_0^{T}\int_{\mathbb{T}^3}\int_{\mathbb{R}^3}f_m(u_m-\xi)\cdot \varphi d\xi dx ds+\int_{\T^3} u^\epi_0\cdot \varphi(0,\cdot)dx,
\end{split}
\end{equation}we can recover the following ones:
\begin{equation}
\begin{split}
\label{weak formulation for NS part}
&\int_0^{T}\int_{\mathbb{T}^3}(-u\cdot \partial_t\varphi+(u\cdot \nabla) u\cdot \varphi+\nabla u: \nabla \varphi)dxds\\
&=-\int_0^{T}\int_{\mathbb{T}^3}\int_{\mathbb{R}^3}f(u-\xi)\cdot \varphi d\xi dx ds+\int_{\T^3} u^\epi_0\cdot \varphi(0,\cdot)dx,
\end{split}
\end{equation}
for any test function $\varphi\in C^\infty([0,T]\times \mathbb{T}^3)$ and ${\rm div}\varphi=0$.

 Similarly, thanks to \eqref{strong convergence of velocity}-\eqref{E4.53}, as $m\to\infty,$ from the weak formulation of the  kinetic equation
\begin{equation}\label{E4.38-}
\begin{split}
&-\int_0^{T}\int_{\mathbb{T}^3}\int_{\mathbb{R}^3}f_m(\phi_t+\xi\cdot \nabla_x \phi +(u_m-\xi)\cdot \nabla_{\xi}\phi)d\xi dx ds\\
=&\int_{\mathbb{T}^3}\int_{\mathbb{R}^3}f^\epi_0\phi(0,x,\xi)d\xi dx
+\lambda\int_0^{T}\int_{\mathbb{T}^3}\int_{\mathbb{R}^3}f_m\phi d\xi dx ds\\
&
-\lambda\int_0^{T}\int_{\mathbb{T}^3}\int_{\mathbb{R}^3}\int_{\mathbb{R}^3}T(\xi,\xi^\prime)f_m(t,x,\xi^\prime)\phi d\xi d\xi^\prime dx ds,
\end{split}
\end{equation}
we can recover
\begin{equation}\label{E4.38}
\begin{split}
&-\int_0^{T}\int_{\mathbb{T}^3}\int_{\mathbb{R}^3}f(\phi_t+\xi\cdot \nabla_x \phi +(u-\xi)\cdot \nabla_{\xi}\phi)d\xi dx ds\\
=&\int_{\mathbb{T}^3}\int_{\mathbb{R}^3}f^\epi_0\phi(0,x,\xi)d\xi dx
+\lambda\int_0^{T}\int_{\mathbb{T}^3}\int_{\mathbb{R}^3}f\phi d\xi dx ds\\
&
-\lambda\int_0^{T}\int_{\mathbb{T}^3}\int_{\mathbb{R}^3}\int_{\mathbb{R}^3}T(\xi,\xi^\prime)f(t,x,\xi^\prime)\phi d\xi d\xi^\prime dx ds,
\end{split}
\end{equation}

\noindent for any test function  $\phi\in C^\infty([0,T]\times\mathbb{T}^3\times\mathbb{R}^3)$.\\

\noindent The last task of this step is to study the limit of the energy inequality as  $m$ goes to infinity. In particular, we can state this limit in the following lemma.
 \begin{lemma}
 \label{lemma of energy inequality}
 The following energy inequality holds as $m\to\infty$:
\begin{equation}\label{E3.19--}
\begin{split}
&\frac{1}{2}\int_{\mathbb{T}^3}|u|^2dx+\int_{\mathbb{T}^3}\int_{\mathbb{R}^3}f(1+\frac{1}{2}|\xi|^2)d\xi dx+\int_0^T \int_{\mathbb{T}^3}|\nabla_x u|^2dx ds\\
&+\int_0^T \int_{\mathbb{T}^3}\int_{\mathbb{R}^3}f|u-\xi|^2d\xi dx  ds\\
\leq &\frac{1}{2}\int_{\mathbb{T}^3}|u_0^\epi|^2dx+\int_{\mathbb{T}^3}\int_{\mathbb{R}^3}f_0^\epi(1+\frac{1}{2}|\xi|^2)d\xi dx.
\end{split}
\end{equation}
\end{lemma}
\begin{proof}
Here,
 we only focus on the most challenging term \begin{equation}
 \label{most difficulty term}\int_0^T \int_{\mathbb{T}^3}\int_{\mathbb{R}^3}f_m|u_m-\xi|^2d\xi dx  ds,
 \end{equation}
when
  studying the limit of \begin{equation*}
\begin{split}
&\frac{1}{2}\int_{\mathbb{T}^3}|u_m|^2dx+\int_{\mathbb{T}^3}\int_{\mathbb{R}^3}f_m(1+\frac{1}{2}|\xi|^2)d\xi dx+\int_0^T \int_{\mathbb{T}^3}|\nabla_x u_m|^2dx ds\\
&+\int_0^T \int_{\mathbb{T}^3}\int_{\mathbb{R}^3}f_m|u_m-\xi|^2d\xi dx  ds\\
\leq & \frac{1}{2}\int_{\mathbb{T}^3}|u_0^\epi|^2dx+\int_{\mathbb{T}^3}\int_{\mathbb{R}^3}f_0^\epi(1+\frac{1}{2}|\xi|^2)d\xi dx,
\end{split}
\end{equation*}
as $m\to\infty.$
For that purpose,   as in \cite{Hamdache}, we rewrite \eqref{most difficulty term} as follows:
\beq
\begin{split}
\int_0^T \int_{\mathbb{T}^3}\int_{\mathbb{R}^3}f_m|u_m-\xi|^2d\xi dx  ds&=\int_0^T \int_{\mathbb{T}^3}\int_{\mathbb{R}^3}f_m |u_m|^2d\xi dxds\\
-& 2\int_0^T \int_{\mathbb{T}^3}u_m\cdot\int_{\mathbb{R}^3}f_m\xi d\xi dxds +\int_0^T \int_{\mathbb{T}^3} m_2f_mdxds.\\
\end{split}
\eeq
By \eqref{E3.37}$_1$($ f_m \rightharpoonup f,  weakly*\ in\  L^\infty(0,T;L^\infty(\T^3\times\R^3)$) and \eqref{strong convergence of velocity},  Fatou's Lemma yields
\beq\notag
 \int_0^T \int_{\mathbb{T}^3}\int_{\mathbb{R}^3}|u|^2 f \mathbf{1}_{|\xi|\leq l} dxd\xi ds\leq  \liminf\limits_{m\rightarrow \infty}\int_0^T \int_{\mathbb{T}^3}\int_{\mathbb{R}^3}f_m |u_m|^2d\xi dxds,
\eeq
where $l>0$ is any positive number.
Letting $l\to\infty$, we can apply the monotone convergence theorem to obtain
\begin{equation}
\label{11}
 \int_0^T \int_{\mathbb{T}^3}\int_{\mathbb{R}^3}|u|^2 f  dxd\xi ds\leq  \liminf\limits_{m\rightarrow \infty}\int_0^T \int_{\mathbb{T}^3}\int_{\mathbb{R}^3}f_m |u_m|^2d\xi dxds.
\end{equation}
Thanks to  \eqref{strong convergence of velocity} and \eqref{E4.53}, we have
\begin{equation}
\label{convergence}
\int_0^T \int_{\mathbb{T}^3}u_m\cdot \int_{\mathbb{R}^3}f_m\xi d\xi dxds \rightarrow \int_0^T \int_{\mathbb{T}^3}u\cdot\int_{\mathbb{R}^3}f\xi d\xi dxds,\  as\   m\rightarrow \infty.
\end{equation}
Note that by \eqref{E4.54},   Fatou's lemma yields
\begin{equation}
\label{22}
\int_0^T \int_{\mathbb{T}^3} m_2fdxds \leq \liminf\limits_{m\rightarrow \infty} \int_0^T \int_{\mathbb{T}^3} m_2f_mdxds.
\end{equation}
By \eqref{11}-\eqref{22}, we have
\beq
\int_0^T \int_{\mathbb{T}^3}\int_{\mathbb{R}^3}f|u-\xi|^2d\xi dx  ds \leq \liminf\limits_{m\rightarrow \infty} \int_0^T \int_{\mathbb{T}^3}\int_{\mathbb{R}^3}f_m|u_m-\xi|^2d\xi dx  ds,\  as\   m\rightarrow \infty.
\eeq
Thus, we are able to show \eqref{E3.19--}.
\end{proof}
\bigbreak

By \eqref{weak formulation for NS part}, \eqref{E4.38}, and Lemma \ref{lemma of energy inequality},  we have the following existence result of weak solutions for all time $t>0$:
\begin{proposition}\label{Pro3.2}
For any  $T>0$, under the assumption of Theorem \ref{Thm 1}, there exists a  weak solution $(u_\epi, f_\epi)$   of the following problem
\begin{equation}\label{E3.181}
\left\{\begin{array}{l}
\partial_t f_\epi
  +
  \xi\cdot \nabla_x f_\epi+{\rm div}_\xi((u_\epi-\xi)f_\epsilon)=-\lambda f_\epi+\lambda\int_{\mathbb{R}^3}T(\xi,\xi^\prime)f_\epi(t,x,\xi^\prime)d\xi^\prime,\\[2mm]
\partial_t u_\epi
 +(u_\epi\cdot \nabla_x)u_\epsilon +\nabla_x P_\epi- \Delta_x u_\epi=-\int_{\mathbb{R}^3}(u_\epi-\xi)f_\epsilon  d\xi,    \\[2mm]
 {\rm div}u_\epi=0,  \ \ \ \   \qquad x\in \mathbb{T}^3,\ \xi\in \mathbb{R}^3, \  t\in(0,T), \\[2mm]
 \end{array}
        \right.
\end{equation}
with initial data
\begin{equation}\label{E3.19----}
(f_\epi,u_\epi)|_{t=0}=(f_0^\epsilon(x,\xi),u_0^\epsilon(x)), \ \ \  x\in \mathbb{T}^3,\ \xi\in \mathbb{R}^3.
\end{equation}
Moreover, for any $0\leq t\le T$, the weak solution   $(u_\epi, f_\epi)$ satisfies the following energy inequality:
\begin{equation}\label{E3.19-------}
\begin{split}
&\frac{1}{2}\int_{\mathbb{T}^3}|u_\epi|^2dx+\int_{\mathbb{T}^3}\int_{\mathbb{R}^3}f_\epi(1+\frac{1}{2}|\xi|^2)d\xi dx+\int_0^T \int_{\mathbb{T}^3}|\nabla_x u_\epi|^2dx ds\\
&+\int_0^T \int_{\mathbb{T}^3}\int_{\mathbb{R}^3}f_\epi|u_\epi-\xi|^2d\xi dx  ds\\
\leq &\left(\frac{1}{2}\int_{\mathbb{T}^3}|u_0^\epi|^2dx+\int_{\mathbb{T}^3}\int_{\mathbb{R}^3}f_0^\epi(1+\frac{1}{2}|\xi|^2)d\xi dx\right),
\end{split}
\end{equation}
 and
\beq\label{E3.201}
\|f_\epi(t,x,\xi)\|_{ L^\infty(0,T;L^p(\R^3\times\T^3))}\leq C,\ for\ any\ 1\leq p\leq  \infty,
\eeq
here $C>0$ depends only on the initial data, $\lambda$ and $T$.
\end{proposition}

\bigbreak

\textbf{Step 2. Passing to the limit as $\epsilon \rightarrow 0$.}\\
Next, we aim to pass to the limit to  recover the weak solution to the problem \eqref{E1.1}-\eqref{E1.2} as $\epsilon\rightarrow 0.$
\begin{remark}
\label{remark on last level}
Although we have fixed $\epsilon>0$ in Step 1, estimates \eqref{E4.1}-\eqref{E4.41-} are uniform  with respect to $\epsilon$. Thus, we have the same convergence results for   $(u_{\epsilon},f_{\epsilon})$ to the same  $(u_m,f_m)$.
\end{remark}

To recover the weak solutions of  \eqref{E1.1}-\eqref{E1.2}, we only need to pass to the limit for  $(u_{\epsilon},f_{\epsilon})$ as $\epsilon\rightarrow 0.$ Thanks to Remark \ref{remark on last level}, we can take the same limit process as $m\to \infty$ to handle the limit  in the weak formulation. It allows us to deduce \eqref{E1.8}, \eqref{E1.9} by letting $\epsilon$ go  to zero.

Note that by the restriction on $u_0^\epi$ and $f_0^\epi$ at the beginning of Section 3, we have
$$\frac{1}{2}\int_{\mathbb{T}^3}|u_0^\epi|^2dx+\int_{\mathbb{T}^3}\int_{\mathbb{R}^3}f_0^\epi(1+\frac{1}{2}|\xi|^2)d\xi dx
\to \frac{1}{2}\int_{\mathbb{T}^3}|u_0|^2dx+\int_{\mathbb{T}^3}\int_{\mathbb{R}^3}f_0(1+\frac{1}{2}|\xi|^2)d\xi dx
$$
as $\epsilon\to 0.$ With the help of the convergence of $(u_{\epsilon},f_{\epsilon})$, this allows us to deduce, as $\epsilon\to 0$,
\begin{equation}\label{E4.36}
\begin{split}
&\frac{1}{2}\int_{\mathbb{T}^3}|u|^2dx+\int_{\mathbb{T}^3}\int_{\mathbb{R}^3}f(1+\frac{1}{2}|\xi|^2)d\xi dx+\int_0^T \int_{\mathbb{T}^3}|\nabla_x u|^2dx ds\\
&+\int_0^T \int_{\mathbb{T}^3}\int_{\mathbb{R}^3}f|u-\xi|^2d\xi dx  ds\\
\leq & \frac{1}{2}\int_{\mathbb{T}^3}|u_0|^2dx+\int_{\mathbb{T}^3}\int_{\mathbb{R}^3}f_0(1+\frac{1}{2}|\xi|^2)d\xi dx.
\end{split}
\end{equation}
Thus, we have completed the proof of our main result.

\section*{Acknowledgement}
 Yao   was supported by the National
Natural Science Foundation of China $\#$11571280, 11331005, FANEDD $\#$201315  and China Scholarship Council.

\bibliographystyle{plain}

\end{document}